\documentclass[10pt,oneside,a4paper,reqno]{amsart}  
\usepackage{mathrsfs}

\usepackage{amssymb}
\usepackage[colorlinks=true, linkcolor=magenta, citecolor=cyan, urlcolor=blue]{hyperref}
\usepackage{bbm}

\numberwithin{equation}{section}

\numberwithin{equation}{section}\newtheorem{theorem}{Theorem}[section]
\newtheorem{corollary}[theorem]{Corollary}\newtheorem{lemma}[theorem]{Lemma}

\newtheorem{proposition}[theorem]{Proposition}\theoremstyle{remark}
\newtheorem{remark}{Remark}[section]
\theoremstyle{definition}
\newtheorem{definition}[theorem]{Definition}

\newcommand{\p}{\widetilde{p}}
\newcommand{\q}{\widetilde{q}}

\newcommand{\Rtre}{\mathbb{R}^{3}}
\newcommand{\RT}{\mathbb{R}^{3}}
\newcommand{\Rpiu}{\mathbb{R}^{+}}

\title[Angular Integrability and Navier--Stokes equation]
{On the regularity set and angular integrability for the 
Navier--Stokes equation}
\date{\today}    

\makeindex

\author{Piero D'Ancona}
\address{Piero D'Ancona: 
SAPIENZA --- Universit\`a di Roma,
Dipartimento di Matematica, 
Piazzale A.~Moro 2, I-00185 Roma, Italy}
\email{dancona@mat.uniroma1.it}

\author{Renato Luc\`a}
\address{Renato Luc\`a: 
Instituto de Ciencias Matematicas, 
Consejo de Investigaciones Cientificas, 
C. Nicolas Cabrera 13-15, 28049 Madrid, Spain}
\email{renato.luca@icmat.es}

\thanks{The authors are partially supported by the
Italian Project FIRB 2012 ``Dispersive
dynamics: Fourier Analysis and Variational Methods''.
The second author is supported by the
ERC grant 277778 and MINECO grant SEV-2011-0087 (Spain)}


\begin{document}
\begin{abstract}
  We investigate the size of the regular set for suitable weak
  solutions of the Navier--Stokes equation, in the sense of
  Caffarelli--Kohn--Nirenberg \cite{CKN}.
  We consider initial data
  in weighted Lebesgue spaces with mixed radial-angular 
  integrability, and we
  prove that the regular set increases if the data
  have higher angular integrability, invading
  the whole half space $\{t>0\}$ in an appropriate limit.
  In particular, we obtain that if the
  $L^{2}$ norm with weight $|x|^{-\frac12}$ of the data
  tends to 0, the regular set invades $\{t>0\}$;
  this result improves Theorem D of \cite{CKN}.
\end{abstract}
\maketitle

\section{Introduction and main results}

We consider the Cauchy problem for the Navier--Stokes
equation on $\Rpiu \times \Rtre$
\begin{equation}\label{CauchyNS}
\left \{
\begin{array}{rcl}
\partial_{t}u + (u \cdot \nabla) u  -\Delta u & = & -\nabla P  \\
\nabla \cdot u & = & 0 \\
u(x,0) & = & u_{0}(x). 
\end{array}\right.
\end{equation}
describing a viscous incompressible fluid in the
absence of external forces, where as usual
$u$ is the velocity field of the fluid and $P$ the pressure,
and the initial data satisfy the compatibility condition 
$\nabla \cdot u_{0}=0$.
We use the same notation for the norm of scalar,
vector or tensor quantities:
\begin{equation*}
  \textstyle
  \| P \|_{L^{2}}:=(\int P^{2}dx)^{\frac12},
  \qquad
  \|u\|_{L^{2}}^{2}:=\sum_{j}\|u_{j}\|_{L^{2}}^{2},
  \qquad
  \|\nabla u\|_{L^{2}}^{2}:=
  \sum_{j,k}\|\partial_{k} u_{j}\|_{L^{2}}^{2}
\end{equation*}
and we write simply $L^{2}(\mathbb{R}^{3})$
instead of $[L^{2}(\mathbb{R}^{3})]^{3}$,
or $\mathscr{S}'(\Rtre)$ instead of $[\mathscr{S}'(\Rtre)]^{3}$ 
and so on. Regularity of
the global weak solutions constructed in
\cite{Hopf, Ler} is a notorious open problem,
although many partial results exist. 

The case of small data is well understood. 
In the proofs of well posedness for small data,
the equation is regarded as a linear heat equation
perturbed by a small nonlinear term $(u\cdot \nabla)u$,
and the natural approach is a fixed point argument
around the heat propagator. More precisely,
one rewrites the problem in integral form
\begin{equation}\label{IntegralCauchyNS}
  \textstyle
  u  =  e^{t \Delta}u_{0} - 
  \int_{0}^{t}e^{(t-s)\Delta}
  \mathbb{P}\nabla \cdot (u \otimes u)(s) \ ds 
  \qquad
  \mbox{in} \quad \Rpiu \times \Rtre 
\end{equation}
where $\mathbb{P}$ is the Leray projection
\begin{equation*}
\mathbb{P}f := f - \nabla \Delta^{-1}(\nabla \cdot f),
\end{equation*}
and then the Picard iteration scheme is defined by
\begin{equation}\label{PicSequence}
  \textstyle
  u_{1}  := e^{t\Delta}u_{0}, \qquad 
  u_{n} := e^{t\Delta}u_{0} - 
  \int_{0}^{t}e^{(t-s)\Delta}
  \mathbb{P}\nabla \cdot (u_{n-1} \otimes u_{n-1})(s) \ ds.
\end{equation}
Once the velocity is known the pressure can be recovered at each time by $P = - \Delta^{-1} \nabla \otimes \nabla (u \otimes u)$.
Small data results fit in the following abstract framework:

\begin{proposition}[\cite{Lem}]\label{PicardThm}
Let 
$X \subset \bigcap_{s<\infty} L^{2}_{t}L^{2}_{uloc, x}((0,s) \times \Rtre)$\footnote{The space 
$L^{2}_{uloc}$ consists of the functions that are uniformly locally square-integrable (see \cite{Lem} Definition 11.3). The operator (\ref{Def:BilinearOperator}) is well-defined on
$\bigcap_{s<\infty} L^{2}_{t}L^{2}_{uloc, x}((0,s) \times \Rtre) \times \bigcap_{s<\infty} L^{2}_{t}L^{2}_{uloc, x}((0,s) \times \Rtre)$. 
We refer to \cite{Lem}, Chapter 11, for more details.} 
be a Banach space such that the bilinear form
\begin{equation}\label{Def:BilinearOperator}
\textstyle
B(u,v) := \int_{0}^{t}e^{(t-s)\Delta}
\mathbb{P}\nabla \cdot (u \otimes v)(s) \ ds
\end{equation}
is bounded from $X \times X$ to $X$:
$$
\|B(u,v)\|_{X} \leq C_{X} \|u\|_{X}  \|v\|_{X}.
$$
Moreover, let $X_{0} \subset \mathscr{S}'(\Rtre)$ be 
a normed space such that
$e^{t \Delta}:X_{0}\to X$ is bounded:
$$
\|e^{t\Delta} f\|_{X} \leq A_{X_{0},X}  \|f\|_{X_{0}}.
$$
Then for every data $u_{0}$ such that
$\|u_{0}\|_{X_{0}} < 1/ 4 C_{X} A_{X_{0},X}$ the sequence
$u_{n}$ is Cauchy in $X$ and converges to 
a solution $u$ of the integral equation (\ref{IntegralCauchyNS}). 
The solution satisfies
$$
\| u \|_{X} \leq 2 A_{X_{0},X} \|u_{0}\|_{X_{0}}.
$$
\end{proposition}

The space $X$ is usually called an 
\emph{admissible (path) space}, while $X_{0}$ is called 
an \emph{adapted space}.
Many adapted spaces $X_{0}$ have been studied:
$L^{3}$ \cite{Kato},
Morrey spaces \cite{Gig3, Tay}, Besov spaces 
\cite{Cannon, Gall, Planc} and several others. 
The largest space in which Picard iteration has 
been proved to converge is $BMO^{-1}$ \cite{Tat}.

A crucial ingredient in the theory is symmetry
invariance.
The natural symmetry of the Navier--Stokes equation is
the translation-scaling
$$
u_{0}(x) 
\mapsto \lambda u_{0} (\lambda (x-x_{0})), \qquad
\lambda \in \mathbb{R}^{+},\ \ x_{0} \in \Rtre,
$$
and indeed all the spaces $X_{0}$ mentioned above
are invariant for this transformation.
On the other hand, in results of local regularity 
a role may be played by some spaces which are scaling but not
translation invariant, like the
weighted $L^{p}$ spaces with norm
$$
\| |x|^{1-\frac3p}  u(x) \|_{L^{p}(\Rtre)}.
$$
When $p=2$ this is the weighted $L^{2}$ space with norm
$\||x|^{-\frac12}u(x)\|_{L^{2}}$, used
in the classical regularity results
of \cite{CKN}. We recall a key definition from that paper:

\begin{definition}
A point $(t_{0},x_{0}) \in \mathbb{R}^{+}
\times \mathbb{R}^{3}$ is 
\emph{regular} for a solution $u(t,x)$ of
(\ref{CauchyNS}) if $u$ is essentially bounded on a neighbourhood of
$(t_{0},x_{0})$.
It follows that $u(t,x)$ is smooth near
$(t_{0},x_{0})$ (see for instance \cite{Ser}). 
A subset of $\mathbb{R}^{+} \times \mathbb{R}^{3}$
is \emph{regular} if all its points
are regular.
\end{definition}

The following result (Theorem D in \cite{CKN})
applies to the special class of \emph{suitable weak solutions},
which are, roughly speaking, solutions with bounded energy;
see the beginning of
Section \ref{sec:prelim} for the precise definition.
We use the notation
\begin{equation*}
  \Pi_{\alpha}:=
  \left\{ (t,x)\in \mathbb{R}^{+}\times \mathbb{R}^{3}
  \ \colon\ 
  t > \frac{|x|^{2}}{\alpha} 
  \right\}
\end{equation*}
to denote the paraboloid of aperture $\alpha$ 
in the upper half space
$\mathbb{R}^{+}\times \mathbb{R}^{3}$; 
note that $\Pi_{\alpha}$ is increasing in $\alpha$.

\begin{theorem}[Caffarelli--Kohn--Nirenberg]\label{CKNSmallData}
  There exists a constant $\varepsilon_{0} > 0$ such that
  the following holds. Let
  $u$ be a suitable weak solution of Problem (\ref{CauchyNS})
  with divergence free initial data
  $u_{0} \in L^{2}(\Rtre)$. If
  \begin{equation*}
    \| |x|^{-1/2} u_{0} \|_{L^{2}(\Rtre)} ^{2}= \varepsilon < 
    \varepsilon_{0}
  \end{equation*}
  then the paraboloid
  \begin{equation*}
    \Pi_{\varepsilon_{0}-\varepsilon}
    \equiv
    \left\{ (t,x) \ : \ t > 
    \frac{|x|^{2}}{\varepsilon_{0} - \varepsilon} \right\}
  \end{equation*} 
  is a regular set.
\end{theorem}

The theorem states that if the weighted $L^{2}$ norm of the
data is sufficiently small, then the solution
is smooth on a certain paraboloid with vertex at the origin.
If the size of the data tends to 0, the regular set increases
and invades a limit paraboloid
$\Pi_{\varepsilon_{0}}$, which is strictly contained in
the half space $t>0$.

It is reasonable to expect that the regular set actually
invades the whole upper half space $t>0$ when the size
of the data tends to 0. This is indeed a special case of
our main result, see Theorem \ref{the:tent} below and
in particular Corollary \ref{cor:CKN}.

However our main goal is a more general investigation
of the influence on the regular set of
additional angular integrability of the data.
We measure angular regularity
using the following mixed norms:
\begin{equation}\label{MixedNorms}
\begin{array}{lcl}
  \|f\|_{L^{p}_{|x|}L^{\p}_{\theta}}& := &
  \left(
    \int_{0}^{+\infty}
    \|f(\rho\ \cdot\ )\|^{p}_{L^{\p}(\mathbb{S}^{2})}
    \rho^{2}d \rho
  \right)^{\frac1p}, \\
  \|f\|_{L^{\infty}_{|x|}L^{\p}_{\theta}}& := &
  \sup_{\rho>0}\|f(\rho\ \cdot\ )\|_{L^{\p}(\mathbb{S}^{2})}.
\end{array}
\end{equation}
The idea of separating radial and angular regularity
is not new; it proved useful especially in the context of
Strichartz estimates and dispersive equations
(see \cite{ChoOzawa09-a},
\cite{DanconaCacciafesta11-a}, 
\cite{FangWang08-a},
\cite{MachiharaNakamuraNakanishi05-a},
\cite{Rogers}
\cite{Sterbenz05-a}).
The $L^{p}_{|x|}L^{\p}_{\theta}$ scale 
includes the usual $L^{p}$ norms when $\p=p$:
\begin{equation*}
  \|u\|_{L^{p}_{|x|}L^{p}_{\theta}} =
  \|u\|_{L^{p}(\mathbb{R}^{3})}.
\end{equation*}
Note also that for
radial functions the value of $\p$ is irrelevant, 
in the sense that
\begin{equation*}
  \text{$u$ radial}\quad\implies\quad 
  \|u\|_{L^{p}_{|x|}L^{\p}_{\theta}} 
  \simeq \| u \|_{L^{p}(\mathbb{R}^{3})}
  \quad \forall p,\p\in[1,\infty]
\end{equation*}
while for generic functions we 
have only\footnote{As usual we write $A \lesssim B$ if there is a constant $C$ 
independent of $A,B$ such that
$A \leq C B$ and $A \simeq B$ if
$A \lesssim B$ and $B \lesssim A$.}
\begin{equation*}
  \|u\|_{L^{p}_{|x|}L^{\p}_{\theta}} 
  \lesssim \|u\|_{L^{p}_{|x|}L^{\p_{1}}_{\theta}}
  \quad\text{if}\quad \p\le\p_{1}.
\end{equation*}
With respect to scaling, the mixed radial-angular 
norm $L^{p}_{|x|}L^{\p}_{\theta}$ behaves like $L^{p}$
and in particular we have for all 
$\widetilde{p}\in[1,\infty]$ and all $\lambda>0$
\begin{equation*}
  \| |x|^{\alpha} \lambda u_{0} 
  (\lambda x)\|_{L^{p}_{|x|}L^{\p}_{\theta}} =
  \| |x|^{\alpha} u_{0} (x)\|_{L^{p}_{|x|}L^{\p}_{\theta}}, 
  \qquad
  \mbox{provided} 
  \qquad 
  \alpha = 1- \frac{3}{p}. 
\end{equation*}
As a first application, we show that for
initial data with small 
$\||x|^{\alpha}u_{0}\|_{L^{p}_{|x|}L^{\p}_{\theta}}$ norm
and
$
\p \geq 2p/(p-1),
$
the problem has a global smooth solution. 
As we prove in Section \ref{sec:prelim},
this norm controls the
$B^{-1 + 3/q}_{q, \infty}$ norm (for $q$ large enough), 
and this space is embedded in $BMO^{-1}$,
thus the existence part in Theorem \ref{WeighKato}
could be deduced 
from the more general results in \cite{Cannon,  Tat, Planc}.
However, the quantitative estimate (\ref{NotInfty})
is new for such initial data,
and it will be a crucial tool for
the proof of our main Theorem \ref{the:tent}.

\begin{theorem}\label{WeighKato}
  Let 
  $1 < p < 5$,  $\p \geq 2p/(p-1)$, 
  $\alpha=1-3/p$
  and let $u_{0} \in L^{p}_{|x|^{\alpha p}d|x|}L^{\p}_{\theta}$ 
  be divergence free. Moreover, let
  \begin{equation}\label{AssumptOnQWK}
  \begin{array}{lcl}
  \frac{2p}{p-1} \leq q < \infty & \mbox{if} & 1 < p \leq 2
  \\
  \frac{2p}{p-1} \leq q < \frac{3p}{p-2} & \mbox{if} & 2 \leq p \leq 3 
  \\   
  p < q < \frac{3p}{p-2} & \mbox{if} & 3 \leq p < 5
  \end{array}
  \end{equation}
  and
  \begin{equation}\label{ScalingAssumptWK}
  \frac{2}{r} + \frac{3}{q} = 1.
  \end{equation}    
  Then there exists an
  $\bar{\varepsilon}=\bar{\varepsilon}(p,\p,q)>0$
  such that, if
  \begin{equation}\label{SmallAssInGlobalThm}
    \||x|^{\alpha}u_{0}\|_{L^{p}_{|x|}L^{\p}_{\theta}}
    < 
    \bar{\varepsilon},
  \end{equation}
  Problem (\ref{IntegralCauchyNS}) has a unique 
  global smooth solution $u$
  satisfying\footnote{Here and in the following we use the notation
$
  \| f \|_{XYZ} := \|\|\| f \|_{Z}\|_{Y}\|_{X}
$ 
for nested norms. When we write $\| u \|_{L^{r}_{t}L^{q}_{x}}$ we mean that the inegration is extended to all the times $t >0$.}
  \begin{equation}\label{NotInfty} 
    \| u \|_{L^{r}_{t}L^{q}_{x}} 
    \le
    \bar{C}
    \||x|^{\alpha} u_{0} \|_{L^{p}_{|x|}L^{\p}_{\theta}}.
  \end{equation}
  for some constant $\bar{C}=\bar{C}(p,\widetilde{p},q)$
  independent of $u_{0}$.
\end{theorem}

In the following we shall need only the special case
corresponding to the choice
\begin{equation*}
  p=2,
  \qquad
  \widetilde{p}=4,
  \qquad
  q=4.
\end{equation*}
Thus, using the notations
\begin{equation}\label{eq:ep1C1}
  \varepsilon_{1} := \bar{\varepsilon}(2,4,4),
  \qquad
  C_{1} := \bar{C}(2,4,4),
\end{equation}
we see in particular that for all divergence free initial data 
with
\begin{equation}\label{DefEpsilon4}
  \||x|^{-1/2}  u_{0}\|_{L^{2}_{|x|}L^{4}_{\theta}} 
  <\varepsilon_{1}
\end{equation}
there exists a unique 
global smooth solution $u(t,x)$, which satisfies
the estimate
\begin{equation}\label{BoundOnWPre}
  \| u \|_{L^{8}_{t}L^{4}_{x}} \leq  C_{1} 
  \||x|^{-1/2} u_{0} \|_{L^{2}_{|x|}L^{4}_{\theta}}.
\end{equation}

To prepare for our last result, we introduce the notations
\begin{equation*}
  \textstyle
  \theta _{1}(\widetilde{p})
  :=(\frac{2 \p - 4}{4 - \p} )^{1 -  \p / 4},
  \qquad
  \theta_{2}(\widetilde{p})
  := (\frac{2 \p - 4}{4 - \p} )^{1 -  \p / 2},
  \qquad
  \widetilde{p}\in(2,4).
\end{equation*}
It is easy to check that $\theta_{1},\theta_{2}\in[0,1]$
and actually
\begin{equation}\label{RoughDefinition1}
  \lim_{\widetilde{p} \rightarrow 2^{+}} \theta_{1}=0, 
  \qquad \lim_{\widetilde{p} \rightarrow 4^{-}} \theta_{1}=1,
\end{equation}
\begin{equation}\label{RoughDefinition2}
  \lim_{\widetilde{p} \rightarrow 2^{+}} \theta_{2}=1, 
  \qquad \lim_{\widetilde{p} \rightarrow 4^{-}} \theta_{2} = 0.
\end{equation}
Thus we may set $\theta_{1}(2)=0$, $\theta_{2}(2)=1$.
We also define the norm
\begin{equation}\label{eq:interpnorm}
  [u_{0}]_{\widetilde{p}}:=
  \||x|^{-\frac{2}{\widetilde{p}}}u_{0}\|^{\frac{\widetilde{p}}{2}-1}
  _{L^{\widetilde{p}/2}_{|x|}L^{\widetilde{p}}_{\theta}}
  \||x|^{-\frac{1}{\widetilde{p}}}u_{0}\|^{2-\frac{\widetilde{p}}{2}}
  _{L^{\widetilde{p}}_{x}}.
\end{equation}
Note the following facts:
\begin{enumerate}
  \item It is easy to construct initial data such that
  $[u_{0}]_{\widetilde{p}}$ is arbitrarily small while 
  $\| u_{0} \|_{BMO^{-1}}$ is arbitrarily large.
  Indeed, fix a test function 
  $\phi\in C^{\infty}_{c}(\mathbb{R}^{3})$ and
  denote with $\phi_{K}(x):=\phi(x-K \xi)$
  its translate in the direction $\xi$ for some $|\xi|=1$
  and $K>1$; we have obviously
  \begin{equation*}
    \||x|^{-\frac{1}{\widetilde{p}}}\phi_{K}\|
    _{L^{\widetilde{p}}_{x}}
    \simeq
    K^{-\frac{1}{\widetilde{p}}}
  \end{equation*}
  since the $L^{\widetilde{p}}_{x}$ norm is translation 
  invariant. On the other hand, if the support of $\phi$
  is contained in a sphere $B(0,R)$, we have
  \begin{equation*}
    \textstyle
    \||x|^{-\frac{2}{\widetilde{p}}}\phi_{K}\|^{\widetilde{p}/2}
         _{L^{\widetilde{p}/2}_{|x|}L^{\widetilde{p}}_{\theta}}
    =
    \int_{0}^{+\infty}(\int_{\mathbb{S}^{2}}
    |\phi(\theta \rho-K \xi)|^{\widetilde{p}}dS_{\theta})^{\frac12}
    \rho d \rho
    \lesssim
    \int_{K-R}^{K+R}K^{-1}
    \rho d \rho
    \simeq 1
  \end{equation*}
  and we obtain
  \begin{equation*}
    [\phi_{K}]_{\widetilde{p}}
    \lesssim
    (1)^{\frac{\widetilde{p}}{2}-1}
    (K^{-\frac{1}{\widetilde{p}}})^{2-\frac{\widetilde{p}}{2}}
    =K^{\frac12-\frac{2}{\widetilde{p}}}.
  \end{equation*}
  Thus, by the translation invariance of $BMO^{-1}$,
  we conclude that if $\widetilde{p}\in[2,4)$
  \begin{equation}\label{eq:normto0}
    [\phi_{K}]_{\widetilde{p}}\to 0
    \ \ \ \text{while}\ \ \ 
    \|\phi_{K}\|_{BMO^{-1}}=const
    \ \ \ \text{as}\ \ \ 
    K\to \infty.
  \end{equation}

  \item In the limit cases $\widetilde{p}=2$ and 
  $\widetilde{p}=4$ we have simply
  \begin{equation}\label{Consistent1}
    [u_{0}]_{2}  = 
    \| |x|^{-1/2} u_{0}\|_{L^{2}_{x}},
    \qquad
    [u_{0}]_{4}  =
     \| |x|^{-1/2} u_{0}\|_{L^{2}_{|x|}L^{4}_{\theta}} 
  \end{equation}
  and actually the $[\ \cdot\ ]_{\widetilde{p}}$ norm
  arises as an interpolation norm between the two cases
  (see \eqref{Omega1}, \eqref{Gamma1} and \eqref{eq:decomp}
  below).
\end{enumerate}

We can now state our main result, which interpolates
between Theorems \ref{CKNSmallData} and \ref{WeighKato}:

\begin{theorem}[]\label{the:tent}
  There exists a constant $\delta>0$ such that the
  following holds. Let $u$ be a suitable weak solution
  of Problem \eqref{CauchyNS} with divergence free initial data
  $u_{0}\in L^{2}(\mathbb{R}^{3})$,
  and let $\widetilde{p}\in[2,4)$ and $M>1$.

  If the norm $[u_{0}]_{\widetilde{p}}$ of the initial data
  satisfies
  \begin{equation}\label{eq:assdata}
    \theta_{1}\cdot [u_{0}]_{\widetilde{p}}\le\delta,
    \qquad
    \theta_{2} \cdot [u_{0}]_{\widetilde{p}}\le\delta e^{-4M^{2}}
  \end{equation}
  then the paraboloid
  \begin{equation}\label{eq:parabreg}
    \Pi_{M \delta}:=
    \left\{
    (t,x)\in \mathbb{R}^{+}\times \mathbb{R}^{3}
    \ \colon\ 
    t>\frac{|x|^{2}}{M\delta}
    \right\}
  \end{equation}
  is a regular set for $u(t,x)$.
\end{theorem}

The result can be interpreted as follows. Since
$\theta_{2}(\widetilde{p})\to0$ as $\widetilde{p}\to4$,
we can choose $\widetilde{p}=\widetilde{p}_{M}$ as a function
of $M$ in such a way that
\begin{equation*}
  e^{4M^{2}}\cdot\theta_{2}(\widetilde{p}_{M})\to0
  \ \ \text{as}\ \ 
  M \to +\infty.
\end{equation*}
Of course we have $\widetilde{p}_{M}\to4^{-}$ as $M \to +\infty$.
Then from the theorem it follows that, for all sufficiently
large $M$,
\begin{equation*}
  [u_{0}]_{\widetilde{p}_{M}}\le\delta
  \quad\implies\quad
  \Pi_{M \delta}
  \ \text{is a regular set for $u$}.
\end{equation*}
In other words, if we take $M\to+\infty$ and
the norm $[u_{0}]_{\widetilde{p}_{M}}$ is less than $\delta$,
then the regular set invades the whole half space $t>0$.
Note that, as remarked above, the $[u_{0}]_{\widetilde{p}_{M}}$
norm can be small even if the $BMO^{-1}$ norm of $u_{0}$ is large.

Even in the special case $\widetilde{p}=2$,
which is covered by Theorem D of \cite{CKN}, the result gives
some new information on the regular set. Indeed, for
$\widetilde{p}=2$ we have 
$\theta_{1}=0$, $\theta_{2}=1$, and we obtain:

\begin{corollary}\label{cor:CKN}
  There exists a constant $\delta>0$ such that for any
  suitable weak solution $u$ with divergence free
  initial data $u_{0} \in L^{2}(\Rtre)$, and for every $M>1$, if the
  initial data satisfy
  \begin{equation*}
    \| |x|^{-1/2} u_{0}\|_{L^{2}_{x}}\le \delta e^{-4M^{2}}
  \end{equation*}
  then the paraboloid $\Pi_{M \delta}$ is a regular set
  for $u$.
\end{corollary}

Thus, taking $M\to +\infty$, we see
that if the weighted $L^{2}$ norm of the data is sufficiently
small, then the regular set invades the whole half space
$t>0$, as claimed above.

The rest of the paper is organized as follows. 
In Section \ref{sec:prelim} we collect the necessary tools,
in particular
we recall the fundamental 
Caffarelli--Kohn--Nirenberg regularity criterion 
from \cite{CKN}; in Section \ref{sec:prelimBis} 
we prove Theorem \ref{WeighKato}, and
Section \ref{sec:proof} is devoted to the proof
of Theorem \ref{the:tent}.


\goodbreak

\section{Preliminaries}\label{sec:prelim}
\setcounter{equation}{0}

We recall some definitions from \cite{CKN}.

\begin{definition}
  Let $u_{0} \in L^{2}(\Rtre)$. The couple $(u,P)$ is a 
  \emph{suitable weak solution} 
  of Problem (\ref{CauchyNS}) 
  if\footnote{This definition of suitable weak solutions is appropriate to work with the initial datum $u_{0}$. For more details compare the Sections 2 and 7 of \cite{CKN}.}
  \begin{enumerate}
    \item $(u,P)$ satisfies (\ref{CauchyNS}) 
    in the sense of distributions;    
    \item $u(t) \rightarrow u_{0}$ weakly in $L^{2}$ as $t \rightarrow 0$;    
    \item for some constants $E_{0}, E_{1}$
    $$
      \int_{\Rtre} |u|^{2}(t) \ dx \leq E_{0}, 
      $$
    for all $t > 0$ and
     $$
     \int\int_{\mathbb{R}^{+} \times \Rtre} |\nabla u|^{2} \ dt dx \leq E_{1};
     $$
    \item for all non negative $\phi \in C^{\infty}_{c}([0, \infty) \times \Rtre)$
    and for all $t > 0$
    \begin{eqnarray}\label{GenEnIneq}
    & & \int_{\Rtre} |u|^{2} \phi (t) + 2 \int_{0}^{t}\int_{\Rtre} | \nabla u |^{2} \phi
    \\ \nonumber    
    & \leq & \int_{\Rtre} |u_{0}|^{2} \phi (0) + \int_{0}^{t} \int_{\Rtre} |u|^{2} 
    (\phi_{t} + \Delta \phi) 
    +\int_{0}^{t}\int_{\Rtre} (|u|^{2} + 2P)u \cdot \nabla \phi.
    \end{eqnarray}        
  \end{enumerate} 
\end{definition}

Suitable weak solutions are known to exist for all $L^{2}$ initial 
data, see \cite{S2} or the Appendix in \cite{CKN}.
Such solutions are also $L^{2}$-weakly continous as functions 
of time (see \cite{Temam}, pp. 281--282), namely
\begin{equation}
\int_{\mathbb{R}^{3}} u(t,x) w(x) \ dx
\rightarrow
\int_{\mathbb{R}^{3}} u(t',x) w(x) \ dx
\end{equation}
for all $w \in L^{2}(\Rtre)$ as $t \rightarrow t'$ 
($t,t' \in [0,+\infty)$);
thus it makes sense to impose the initial condition (2).

Next we define the \emph{parabolic cylinder} 
of radius $r$ and top point $(t,x)$ as
\begin{equation*}
  Q_{r}(t,x) := 
  \left\{ (s,y): \quad |x-y| < r,\ \ t- r^{2} < s < t
  \right\}
\end{equation*}
while the \emph{shifted parabolic cylinder} is
\begin{equation*}
  Q^{*}_{r}(t,x) := Q_{r}(t + r^{2}/8,x) \equiv
  \left\{ (s,y): \quad |x-y| < r,\ \ t- 7r^{2}/8 < s < t+r^{2}/8
  \right\}
\end{equation*}
The crucial regularity result in \cite{CKN} ensures that:

\begin{lemma}
  There exists 
  an absolute constant $\varepsilon^{*}$ such that if $(u,P)$
  is a suitable weak solution of (\ref{CauchyNS}) and
  \begin{equation}\label{CKNConditionSB}
  \limsup_{r \rightarrow 0} \frac{1}{r}
  \int\int_{Q^{*}_{r}(t,x)} |\nabla u|^{2} \leq \varepsilon^{*},
  \end{equation}
  then $(t,x)$ is a regular point.
\end{lemma}

We shall make frequent use of
the following interpolation inequality from
\cite{CaffarelliKohnNirenberg84-a}
(see also \cite{DL, DenapoliDrelichmanDuran10-a} for extensions
of the inequality):

\begin{lemma}
  Assume that 
  \begin{enumerate}
    \item $r \geq 0$, $0 < a \leq 1$, 
    $\gamma < 3/r$, $\alpha < 3/2$, $\beta < 3/2$;
    \item $-\gamma + 3/r = a(-\alpha + 1/2) + (1-a)(-\beta + 3/2)$;
    \item $ a \alpha + (1-a)\beta \leq \gamma   $;
    \item when $- \gamma + 3/r = -\alpha + 1/2$,
    assume also that $\gamma \leq a (\alpha +1) + (1-a)\beta$.
  \end{enumerate}
  Then
  \begin{equation}\label{CKNInequalityCKNVersion}
    \| \sigma_{\nu}^{\gamma} u \|_{L^{r}(\mathbb{R}^{3})} \leq C
    \| \sigma_{\nu}^{\alpha} \nabla u \|^{a}_{L^{2}(\mathbb{R}^{3})} 
    \| \sigma_{\nu}^{\beta} u \|^{1-a}_{L^{2}(\mathbb{R}^{3})}, 
  \end{equation}
  where $\sigma_{\nu} := (\nu + |x|^{2})^{-1/2}$, 
  $\nu \geq 0$, with a constant $C$ independent of $\nu$.
\end{lemma}

A key role in the following will be played
by time-decay estimates for convolutions 
with the heat and Oseen 
kernels. It is convenient to introduce the notation
\begin{equation*}
  \Lambda(\alpha, p, \p) := \alpha + \frac{2}{p} - \frac{2}{\p}.
\end{equation*}

\begin{proposition}[\cite{Ren}]\label{PDecayCor}
  Let $1 \leq p \leq q \leq \infty$, 
  $1 \leq \p \leq \q \leq \infty$ and
  \begin{equation}\label{eq:condDL(Heat)}
    \beta > -\frac 3q,\qquad \alpha< 3 - \frac{3}{p}, 
    \qquad
    \Lambda (\alpha,p,\p) \geq \Lambda (\beta,q,\q).
  \end{equation}
  For every multiindex $\eta$,
  \begin{enumerate}
    \item if
    $|\eta| + \frac{3}{p}-\frac{3}{q} + \alpha-\beta \geq0$,
    then
    \begin{equation}\label{PHeatDer}
      \||x|^{\beta} \partial^{\eta} 
      e^{t\Delta}u_{0}\|_{L^{q}_{|x|}L^{\q}_{\theta}} 
      \lesssim \frac{1}{t^{(|\eta| 
      + \frac{3}{p}-\frac{3}{q} + \alpha-\beta)/2}} 
      \| |x|^{\alpha} u_{0}\|_{L^{p}_{|x|}L^{\p}_{\theta}}, 
      \qquad t>0;
    \end{equation}
    \item if
    $1+ |\eta| + \frac{3}{p}-\frac{3}{q} + \alpha-\beta > 0$,
    then
    \begin{equation}\label{PGHeatDer}
      \||x|^{\beta} \partial^{\eta} e^{t\Delta} \mathbb{P} \nabla 
      \cdot F\|_{L^{q}_{|x|}L^{\q}_{\theta}} 
      \lesssim \frac{1}{t^{(1 + |\eta| 
      + \frac{3}{p} -\frac{3}{q} +\alpha -\beta)/2}} 
      \| |x|^{\alpha} F\|_{L^{p}_{|x|}L^{\p}_{\theta}}, 
      \qquad t>0.
    \end{equation} 
  \end{enumerate}
\end{proposition}

An easy consequence of Proposition (\ref{PDecayCor}) 
is the embedding
$$
  L^{p}_{|x|^{\alpha p}d|x|}L^{\p}_{d \theta} 
  \hookrightarrow
  B^{-1 + 3/q}_{q,\infty}
  \quad\text{if}\quad 
  \alpha = 1 - \frac{3}{p}, 
  \qquad \p \geq \frac{2p}{p-1},
  \qquad q \geq \max(p,\p),
$$
which is not needed in the following, but allows to
compare Theorem \ref{WeighKato} with earlier results;
recall also that 
$B^{-1 + 3/q}_{q,\infty}\hookrightarrow BMO^{-1}$ for $q>3$.
Indeed, using
estimate (\ref{PHeatDer}), 
we can write
\begin{eqnarray}\nonumber
  \| e^{t \Delta} \phi \|_{L^{q}(\mathbb{R}^{3})}
  \le
  C t^{-(3/p - 3/q + \alpha)/2}
  \| |x| ^{\alpha}\phi \|_{L^{p}_{|x|}L^{\p}_{\theta}}
  \equiv
  C t^{-(1 - 3/q)/2}
  \| |x|^{\alpha} \phi \|_{L^{p}_{|x|}L^{\p}_{\theta}}
\end{eqnarray}
and then the embedding follows immediately from
the following 'caloric' defininition of 
Besov spaces (see e.g. \cite{Tat}):

\begin{definition}
  A distribution $\phi\in \mathscr{S}'$ belongs to
  $B^{-1 + 3/q}_{q, \infty}(\mathbb{R}^{3})$ ($q>3$) if
  and only if
  \begin{equation}\label{NSBEsovDef}
    \|e^{t \Delta} \phi\|_{L^{q}(\mathbb{R}^{3})} 
    \leq 
    Ct^{-(1 - 3/q)/2}
    \quad\text{for}\quad 
    0<t\le1.
  \end{equation}
  The best constant $C$ in (\ref{NSBEsovDef}) is equivalent
  to the norm
  $\| \phi \|_{B^{-1 + 3/q}_{q, \infty}(\mathbb{R}^{3})}$. 
\end{definition}

We conclude this section with an estimate 
for singular integrals in mixed radial-angular norms.
Let $K \in C^{1}(\mathbb{S}^{2})$
with zero mean value and
\begin{equation}
T f (x):=
\mathrm{PV} 
\int_{\mathbb{R}^{3}} f(x-y)\frac{K(\widehat{y})}{|y|^{n}} \ dy,
\qquad
\widehat{y}=\frac{y}{|y|}.
\end{equation}

\begin{theorem}
Let $1 < p < \infty$, $1 < \p < \infty$. 
Then
\begin{equation}\label{RTinMixedNorm}
\| T f\|_{L^{p}_{|x|}L^{\p}_{\theta}}
\lesssim 
\| f\|_{L^{p}_{|x|}L^{\p}_{\theta}}.
\end{equation}
\end{theorem}

The inequality (\ref{RTinMixedNorm})  
has been recently proved 
by A. C\'{o}rdoba in the case $\p = 2$
(\cite{Cordoba}, Theorem 2.1);
essentially the same argument gives also the other cases. 

\begin{proof}
Consider first the case $p > \p$.
Let $1/q + \p/p = 1$ and denote by $X$ the set
of all $g \in \mathscr{S}(\mathbb{R})$ 
with $\int_{0}^{+\infty}g^{q}(\rho)\rho^{2} d\rho =1$.
Then we can write
\begin{equation*}  
\begin{split}
  \| Tf \|^{\p}_{L^{p}_{|x|}L^{\p}_{\theta}}  
  = &  
  \textstyle
   \left(
  \int_{0}^{+\infty} \left(
  \int_{\mathbb{S}^{2}} |Tf(\rho,\theta)|^{\p} \ dS_{\theta}
  \right)^{\frac{p}{\p}} \rho^{2} \ d\rho
  \right)^{\frac{\p}{p}}
\\
  = & 
  \textstyle
  \sup\limits_{g\in X}
  \int_{0}^{+\infty} 
  \int_{\mathbb{S}^{2}} |Tf(\rho,\theta)|^{\p} g(\rho) \rho^{2} 
  \ dS_{\theta} d\rho
\\
  = & 
  \textstyle
  \sup\limits_{g\in X}
  \int_{\mathbb{R}^{3}} |Tf(x)|^{\p} g(|x|) \ dx.
\end{split}
\end{equation*}
Write $I(f,g):=\int_{\mathbb{R}^{3}} |Tf(x)|^{\p} g(|x|)dx$.
By Proposition 1 in  \cite{CordobaFefferman} we have
\begin{equation*}
  \textstyle
  I(f,g) \lesssim_{s} \int_{\mathbb{R}^{3}} |f(x)|^{\p}
  \left( M g^{s} (x)\right)^{\frac{1}{s}} \ dx,
\end{equation*}
for all $1 < s < \infty$,
where $M$ is the Hardy--Littelwood maximal operator
and $g^{s}(x) = (g(|x|))^{s}$. 
Since $M g^{s}$ is radially symmetric, this can be written
\begin{equation*}
  \textstyle
  I(f,g) 
  \lesssim_{s} 
  \int_{0}^{+\infty} \int_{\mathbb{S}^{2}} |f(\rho,\theta)|^{\p}
  (M g^{s}(\rho))^{\frac{1}{s}} \ \rho^{2} dS_{\theta} d\rho.
\end{equation*}
Now, for $s < q = \frac{p}{p-\p}$,
H\"{o}lder's inequality with exponents $p/\p$, $q$ gives
\begin{equation*}  
\begin{split}
  I(f,g)   
  & \lesssim 
  \textstyle
  \left(
  \int_{0}^{+\infty} \left(
  \int_{\mathbb{S}^{2}} |f(\rho,\theta)|^{\p} dS_{\theta} 
  \right)^{\frac{p}{\p}} \rho^{2} \ d\rho 
  \right)^{\frac{\p}{p}}
  \left(\int_{0}^{+\infty}
  (M g^{s}(\rho))^{\frac{q}{s}} \ \rho^{2} d\rho
  \right)^{\frac{1}{q}}
\\
  & \lesssim 
  \| f \|^{\p}_{L^{p}_{|x|}L^{\p}_{\theta}}
  \| M g^{s} \|^{1/s}_{L^{q/s}(\mathbb{R}^{3})}
  \lesssim 
  \| f \|^{\p}_{L^{p}_{|x|}L^{\p}_{\theta}}
  \| g^{s} \|^{1/s}_{L^{q/s}(\mathbb{R}^{3})}
\\
  & \simeq
  \textstyle
  \| f \|^{\p}_{L^{p}_{|x|}L^{\p}_{\theta}}
  \left(\int_{0}^{+\infty} g^{q}(\rho)\rho^{2} 
  \ d\rho \right)^{\frac{1}{q}}=
  \| f \|^{\p}_{L^{p}_{|x|}L^{\p}_{\theta}}
\end{split}
\end{equation*}
and taking the supremum over all $g\in X$ we get the claim
in the case $p>\p$. The case $p=\p$ is classical, 
and the case $p<\p$ follows by duality.
\end{proof}

Using the continuity of $T$
in weighted Lebesgue spaces (see Stein \cite{Stein})
\begin{equation}\label{SteinIneqClassical}
  \textstyle
  \| |x|^{\beta} Tf \|_{L^{p}(\mathbb{R}^{3})}
  \lesssim \| |x|^{\beta} f\|_{L^{p}(\mathbb{R}^{3})}
  \quad\text{for}\quad 
  1 < p < \infty, \quad 
  -\frac{3}{p}< \beta < 3- \frac{3}{p}
\end{equation}
we can also obtain weighted versions of 
\eqref{RTinMixedNorm}. In particular, by interpolation of 
\begin{equation}
\begin{array}{lll}
(\ref{RTinMixedNorm}) & \mbox{in the case} & (\alpha_{0}, p_{0}, \widetilde{p}_{0}) = (0,2,10)
\\ 
(\ref{SteinIneqClassical}) & \mbox{in the case} & (\alpha_{1}, p_{1}, \widetilde{p}_{1}) = (-4/3,2,2),
\end{array}
\end{equation} 
with $\theta = 3/8$ 
($ \Rightarrow (\alpha_{\theta}, p_{\theta}, \widetilde{p}_{\theta})=(-1/2,2,4)$),
we get
\begin{equation}\label{ContinuityOfP}
\| |x|^{-1/2} Tf \|_{L^{2}_{|x|}L^{4}_{\theta}}
\lesssim \| |x|^{-1/2} f\|_{L^{2}_{|x|}L^{4}_{\theta}}.
\end{equation}

\begin{remark}\label{PContinuityRem}
We denote with $R_{j}$ the Riesz transform in the direction of the 
$j$-th coordinate and $R := (R_{1}, R_{2}, R_{3})$. 
By (\ref{SteinIneqClassical}, \ref{ContinuityOfP}) the
boundedness of $R_{j}$ in $L^{2}(\mathbb{R}^{3},|x|^{-1}dx)$ and
$L^{2}_{|x|}L^{4}_{\theta}(\mathbb{R}^{3},|x|^{-1}dx)$ follows, and so that
of $\mathbb{P} \equiv Id + R \otimes R$.  
\end{remark}

\section{Proof of Theorem \ref{WeighKato}}\label{sec:prelimBis}
\setcounter{equation}{0}

We first need two technical lemmas. By standard machinery,
integral estimates for the heat flow and for the bilinear operator
appearing in the Duhamel representation (\ref{IntegralCauchyNS})
can be deduced by the time-decay estimates
of Proposition \ref{PDecayCor}.

\begin{lemma}[\cite{Ren}]\label{IDecayCor}
  Let
  $\beta > - 3/q$, $ \alpha < 3 - 3/p$,
  $1 \leq \p \leq \q \leq \infty$,
  $1 < r < \infty$ and
  \begin{equation}\label{Cond:EquivTop<r}
  \begin{array}{lcl}
  1 \leq p \leq q \leq \infty & \mbox{if} & (|\eta| + \alpha-\beta)p + 1 < 0,  
  \\
  1 \leq p \leq q < \frac{3p}{(|\eta| + \alpha-\beta)p + 1} & \mbox{if} & 
  (|\eta| + \alpha-\beta)p + 1 \geq 0.
  \end{array}
  \end{equation}   
  Assume further that 
      \begin{equation}\label{OmegaScaling}
    |\eta| + \alpha + \frac{3}{p} = \beta + \frac{3}{q} + \frac{2}{r}, 
    \qquad 
    \Lambda (\alpha,p,\p) \geq \Lambda (\beta,q,\q).
  \end{equation}
  Then for every multiindex $\eta$ we have
  \begin{equation}\label{IHeat}
    \| |x|^{\beta} \partial^{\eta} 
    e^{t\Delta}u_{0}\|_{L^{r}_{t}L^{q}_{|x|}L^{\q}_{\theta}} 
    \lesssim  
    \| |x|^{\alpha} u_{0}\|_{L^{p}_{|x|}L^{\p}_{\theta}}.
  \end{equation}
\end{lemma}

\begin{remark}
Once we have assumed the scaling relation in (\ref{OmegaScaling}), 
it is straighforward to check that the assumption (\ref{Cond:EquivTop<r})
is equivalent to $p<r$.
\end{remark}

\begin{proof}
The family of estimates (\ref{IHeat}) follows by the family
of estimates (\ref{PHeatDer}) and by the
Marcinkiewickz interpolation theorem.
The condition $p<r$, which as remarked above turns out to 
be equivalent to (\ref{Cond:EquivTop<r}),
is necessary in order to apply the Marcinkiewickz theorem 
(see Proposition 3.4 in \cite{Ren} for details).
\end{proof}

\begin{lemma}\label{Lemma:DuhamTermBound}
  Let $3 < q < \infty $, 
    $2 < r < \infty$ satisfying
    $2/r + 3/q =1$.
   Then 
  \begin{equation}\label{DGHeatq=p}
    \left\| \int_{0}^{t} e^{(t-s)\Delta} 
    \mathbbm{P} \nabla 
    \cdot (u \otimes v)(s) \  ds
    \right\|_{L^{r}_{t}L^{q}_{x}} 
    \lesssim
    \|  u \|_{L^{r}_{t}L^{q}_{x}} 
    \| v \|_{L^{r}_{t}L^{q}_{x}}.
   \end{equation}
\end{lemma}

The inequality (\ref{DGHeatq=p}) is well known,
see for instance Theorem 3.1(i) in \cite{Fabes}.
The $L^{r}_{t}L^{q}_{x}$ Lebesgue spaces have been extensively used
in the context of Navier--Stokes equation since \cite{Fabes, Giga}.



Using the previous estimates, it is a simple matter to prove
Theorem \ref{WeighKato}. We follow the scheme 
of the proof of Theorem
20.1(B) in \cite{Lem} and we take advance of the inequalities
(\ref{PHeatDer}, \ref{IHeat}). 

\begin{proof}[Proof of Theorem \ref{WeighKato}]

Let $\p_{G}:= 2p/(p-1)$.
We show that the space 
\begin{equation}\nonumber
X := \left\{ u \ : \ \| u \|_{L^{r}_{t}L^{q}_{x}} < \infty , \ \sup_{t >0} t^{1/2} \| u \|_{L^{\infty}_{x}}(t) < \infty \right\},
\end{equation}
equipped with the norm 
$\| \cdot \|_{X} := \| \cdot \|_{L^{r}_{t}L^{q}_{x}} + \sup_{t >0} t^{1/2} \| \cdot \|_{L^{\infty}_{x}}(t)$,
is an admissible path space with adapted space $X_{0} := L^{p}_{|x|^{\alpha p}d|x|}L^{\p_{G}}_{\theta}$.

The estimate $\| e^{t \Delta} f \|_{X} \lesssim \| f \|_{X_{0}}$
follows indeed by the inequalities (\ref{PHeatDer}, \ref{IHeat});
it is straightforward to check that (\ref{Cond:EquivTop<r}) and $p,\p_{G} \le q$ 
are equivalent\footnote{Except that the value $q =p$ is not allowed in (\ref{AssumptOnQWK}).} to (\ref{AssumptOnQWK}) and that the last assumption in 
(\ref{OmegaScaling}) and in (\ref{eq:condDL(Heat)}) is satisfied because 
$\Lambda(\alpha,p, \p_{G}) = \Lambda(0,q,q) = \Lambda(0,\infty,\infty) = 0$. 
Notice also that the set of $q$ for which the third inequality in
(\ref{AssumptOnQWK}) is satisfied is not empty provided $p < 5$.

It remains to show that $\| B(u,v) \|_{X} \lesssim \| u \|_{X} \| v \|_{X}$.
The bound 
$\| B(u,v) \|_{L^{r}_{t}L^{q}_{x}} \lesssim 
\| u \|_{L^{r}_{t}L^{q}_{x}}  \| v \|_{L^{r}_{t}L^{q}_{x}}$
follows by Lemma \ref{Lemma:DuhamTermBound}. 
In order too estimate 
$\sup_{t>0} t^{1/2} \| B(u,v) \|_{L^{\infty}}(t)$,
we split this quantity as
\begin{equation*}
  \sup_{t>0} t^{1/2} \| B(u,v) \|_{L^{\infty}_{x}}(t) 
  \leq I+II
\end{equation*}
where
\begin{equation*}
  I=
  \textstyle
  \sup_{t>0} t^{1/2} 
  \left\| 
  \int_{0}^{t/2} e^{(t-s)\Delta} 
  \mathbb{P} \nabla \cdot (u \otimes v)(s) \ ds
  \right\|_{L^{\infty}_{x}}
\end{equation*}
\begin{equation*}
  II=
  \textstyle
  \sup_{t>0} t^{1/2}
  \left\|
  \int_{t/2}^{t} e^{(t-s)\Delta} 
  \mathbb{P} \nabla \cdot (u \otimes v)(s) \ ds
  \right\|_{L^{\infty}_{x}},
\end{equation*}
and we use Minkowski inequality and the 
time-decay estimate \eqref{PGHeatDer}. For $I$ we have
\begin{eqnarray}\nonumber
  I & \lesssim & 
  \sup_{t>0} t^{1/2}
  \int_{0}^{t/2}
  \frac{1}{(t-s)^{\left(1 + \frac{3}{q/2}\right)/2}}
  \| u \|_{L^{q}_{x}}
  \| v \|_{L^{q}_{x}}(s) \ ds
  \\ \nonumber
  & \lesssim &
  \sup_{t>0}
  t^{-3/q} \int_{0}^{t/2}
  \| u \|_{L^{q}_{x}}
  \| v \|_{L^{q}_{x}}(s) \ ds
  \\ \nonumber
  & \lesssim &
  \sup_{t>0}
  t^{-3/q}
  \| u \|_{L^{r}_{t}L^{q}_{x}}
  \| v \|_{L^{r}_{t}L^{q}_{x}}
  \left( \int \chi_{[0,t/2]}(s) \ ds \right)^{1 - \frac{2}{r}}
  \\ \nonumber
  & \lesssim & 
  \| u \|_{L^{r}_{t}L^{q}_{x}} 
  \| v \|_{L^{r}_{t}L^{q}_{x}}
  t^{-3/q -2/r +1} 
  = \| u \|_{L^{r}_{t}L^{q}_{x}}
  \| v \|_{L^{r}_{t}L^{q}_{x}}
\end{eqnarray}
while for $II$ we have
\begin{eqnarray}\nonumber
  II & \lesssim & 
  \sup_{t>0} t^{1/2}
  \int_{t/2}^{t}
  \frac{1}{(t-s)^{1/2}}\frac{1}{s}
  \left(s^{1/2} \| u \|_{L^{\infty}_{x}} (s)\right)
  \left(s^{1/2} \| v \|_{L^{\infty}_{x}} (s)\right) \ ds
  \\ \nonumber
  & \lesssim &
  \left(\sup_{t>0} t^{1/2}
  \| u \|_{L^{\infty}_{x}}\right)
  \left(\sup_{t>0} t^{1/2}
  \| v \|_{L^{\infty}_{x}}\right)
  \sup_{t>0}t^{-1/2}
  \int_{t/2}^{t} \frac{1}{(t-s)^{1/2}} \ ds
  \\ \nonumber
  & \lesssim &
  \left(\sup_{t>0} t^{1/2}
  \| u \|_{L^{\infty}_{x}}\right)
  \left(\sup_{t>0} t^{1/2}
  \| v \|_{L^{\infty}_{x}}\right)
  \sup_{t>0}t^{-1/2}
  \left[ (t-s)^{1/2} \right]_{t}^{t/2} 
  \\ \nonumber
  & \lesssim & 
  \left(\sup_{t>0} t^{1/2}
  \| u \|_{L^{\infty}_{x}}\right)
  \left(\sup_{t>0} t^{1/2}
  \| v \|_{L^{\infty}_{x}}\right). 
\end{eqnarray}
Summing up we obtain
\begin{equation*}
  \textstyle
  \| B(u,v) \|_{X} 
  \lesssim \| u \|_{L^{r}_{t}L^{q}_{x}}
  \| u \|_{L^{r}_{t}L^{q}_{x}}
  + \left( \sup_{t>0} t^{1/2} \| u \|_{L^{\infty}_{x}} \right)
  \left( \sup_{t>0} t^{1/2} \| v \|_{L^{\infty}_{x}} \right)
  \lesssim \| u \|_{X} \| u \|_{X} .
\end{equation*}
The existence of a unique solution $u$ to Problem 
\eqref{IntegralCauchyNS} satisfying
\begin{equation}\label{NotInftyBis}
    \| u \|_{L^{r}_{t}L^{q}_{x}} 
    +
    \sup_{t>0} t^{1/2} \| u \|_{L^{\infty}_{x}}(t)    
    \lesssim
    \||x|^{\alpha} u_{0} \|_{L^{p}_{|x|}L^{\p}_{\theta}}
\end{equation}
follows by Proposition \ref{PicardThm} and by the
obvious inequality
\begin{equation}\nonumber
\||x|^{\alpha} u_{0} \|_{L^{p}_{|x|}L^{\p_{G}}_{\theta}} 
\lesssim \||x|^{\alpha} u_{0} \|_{L^{p}_{|x|}L^{\p}_{\theta}}.
\end{equation}
 
Finally, inequality \eqref{NotInftyBis}
implies the boundedness of
the solution $u$ in 
$(\delta ,\infty) \times \Rtre$ for all $\delta > 0$, 
and this implies smoothness of the solution
(see Theorem 3.4 in \cite{Fabes} or 
\cite{Esc, Giga, Ser, Sohr, Struwe, WW}).


\end{proof}

We denote with $BC([0,\infty);L^{2})$ the Banach space of
bounded continuous functions 
$u:[0,\infty)\to L^{2}$
equipped with the norm $\sup_{t\geq0} \| u(t) \|_{L^{2}}$.

\begin{corollary}\label{Corh:SmallExistenceAlsoInL2}
  Assume all the hypotheses of Theorem \ref{WeighKato} are satisfied,
  and in addition assume $u_{0} \in L^{2}(\Rtre)$.
  Then the solution $u(t)$ belongs to
  $BC([0,\infty);L^{2}(\Rtre))$. 
  In particular
  $u$ is a strong solution of \eqref{CauchyNS},
  $u(t) \rightarrow u_{0}$ strongly in $L^{2}(\Rtre)$
  as $t \rightarrow 0$,
  and the energy identity
  $\| u(t) \|^{2}_{L^{2}_{x}} + 2 \| \nabla u \|^{2}_{L^{2}_{t}L^{2}_{x}} 
    = \| u_{0} \|^{2}_{L^{2}}$
  holds for all $t >0$.
\end{corollary}

\begin{proof}
  Let $X, X_{0}$ be the same admissible and adapted spaces
  used in the proof of Theorem \ref{WeighKato}.
  As in that proof, we shall show that the space
  $X \cap BC([0,\infty);L^{2}_{x})$ equipped with the norm 
  $\| \cdot \|_{X} + \| \cdot \|_{L^{\infty}_{t}L^{2}_{x}}$ 
  is an admissible path space with adapted space 
  $X_{0} \cap L^{2}_{x}$ equipped with 
  the norm $\| \cdot \|_{X_{0}} + \| \cdot \|_{L^{2}_{x}}$.

  The estimate 
  $\| e^{t \Delta} f \|_{X \cap BC([0,\infty);L^{2}_{x})} 
  \lesssim \| f \|_{X_{0} \cap L^{2}_{x}}$ again follows 
  by (\ref{PHeatDer}, \ref{IHeat}). Since we have
  already proved $\|B(u,v)\|_{X} \lesssim \| u \|_{X} \| v \|_{X}$, 
  it remains to show that
  $\|B(u,v)\|_{L^{\infty}_{t}L^{2}_{x}} \lesssim  
    \| u \|_{X \cap BC([0,\infty);L^{2}_{x})}
    \| v \|_{X \cap BC([0,\infty);L^{2}_{x})}$.
  By Minkowski inequality and Proposition \ref{PGHeatDer},
\begin{eqnarray}\label{Techn:L2Cont}
\| B(u,v) \|_{L^{\infty}_{t}L^{2}_{x}} 
& \lesssim &
\sup_{t>0} \int_{0}^{t} \frac{1}{(t-s)^{1/2}} \frac{1}{s^{1/2}} 
\left(s^{1/2} \| u \|_{L^{\infty}_{x}} (s) \right) \| v \|_{L^{2}_{x}}(s) \ ds
\\ \nonumber
& \leq & 
\left(\sup_{t>0} t^{1/2} \| u \|_{L^{\infty}_{x}} (t) \right) 
\| v \|_{L^{\infty}_{t}L^{2}_{x}} 
\sup_{t>0}\int_{0}^{t} (t-s)^{-1/2} s^{-1/2} \ ds.
\end{eqnarray}
Since $\int_{0}^{t} (t-s)^{-1/2} s^{-1/2} \ ds \leq C$ with
$C$ independent of $t$, (\ref{Techn:L2Cont}) implies
\begin{equation}\label{Techn:L2ContFinal}
\| B(u,v) \|_{L^{\infty}_{t}L^{2}_{x}}
\lesssim \left(\sup_{t>0} t^{1/2} \| u \|_{L^{\infty}_{x}} (t) \right) 
\| v \|_{L^{\infty}_{t}L^{2}_{x}}
\lesssim 
\| u \|_{X \cap BC([0,\infty);L^{2}_{x})}
\| v \|_{X \cap BC([0,\infty);L^{2}_{x})}.
\end{equation}
These inequalities allow us to
apply Proposition \ref{PicardThm}, and we obtain that
$u \in X \cap BC([0,\infty);L^{2}_{x})$ provided
\begin{equation}\label{SmallnessBefResc}
\| u_{0} \|_{X_{0}} = \| |x|^{\alpha} u_{0} \|_{L^{p}_{|x|}L^{\p}_{\theta}}
+ \| u_{0} \|_{L^{2}_{x}} < \bar{\varepsilon},
\end{equation}
with an $\bar{\varepsilon}$ possibly smaller
than in Theorem \ref{WeighKato}. 
On the other hand, rescaling the initial data
and the solution as
\begin{equation}\label{RescInL2Cont}
u_{0}^{\lambda} = \lambda u_{0}(\lambda x), 
\quad
u^{\lambda} = \lambda u(\lambda^{2} t, \lambda x),
\quad 
\lambda > 0 ;
\end{equation}
we see that all norms remain fixed with the exception of
\begin{equation}
\| u^{\lambda}_{0} \|_{L^{2}_{x}}, 
\
\| u^{\lambda} \|_{L^{\infty}_{t}L^{2}_{x}}
\rightarrow 0
\quad 
\mbox{as}
\quad
\lambda \rightarrow + \infty,
\end{equation}
so that (\ref{SmallnessBefResc}) is satisfied 
by $u_{0}^{\lambda}$, provided
$\| |x|^{\alpha} u_{0} \|_{L^{p}_{|x|}L^{\p}_{\theta}} = \rho < \bar{\varepsilon}$ and $\lambda = \lambda(\rho)$ is large enough. In this way we 
find that $\| |x|^{\alpha} u_{0} \|_{L^{p}_{|x|}L^{\p}_{\theta}} < \bar{\varepsilon}$
implies $u^{\lambda} \in BC([0,\infty);L^{2}_{x})$ and 
hence $u \in BC([0,\infty);L^{2}_{x})$.

In particular we have $u(t) \rightarrow u_{0}$ strongly in $L^{2}(\Rtre)$
as $t \rightarrow 0^{+}$. By this remark, and by the smoothness of $u$,
it follows that $u$ is a strong solution 
of (\ref{CauchyNS}) 
which satisfies the energy identity
\begin{equation}
\| u(t) \|^{2}_{L^{2}_{x}}
+ 2 \| \nabla u \|^{2}_{L^{2}_{t}L^{2}_{x}}
= \| u_{0} \|^{2}_{L^{2}},
\qquad t \geq 0.
\end{equation}
 \end{proof}

\begin{remark}
It is straightforward to check that the solution
constructed in Corollary \ref{Corh:SmallExistenceAlsoInL2}
is unique in the class 
of the weak solutions satisfying the energy inequality. 
More precisely, if 
$u'$ is another weak solution of (\ref{CauchyNS}) 
satisfying
\begin{equation}
\| u'(t) \|^{2}_{L^{2}_{x}}
+ 2 \| \nabla u'(t) \|^{2}_{L^{2}_{t}L^{2}_{x}}
\leq \| u_{0} \|^{2}_{L^{2}},
\qquad
t > 0,
\end{equation}
the boundedness condition $\| u \|_{L^{r}_{t}L^{q}_{x}} < \infty$
allows to apply the well known Prodi-Serrin uniqueness criterion
 (\cite{Prodi,Ser2}) to conclude $u=u'$.
\end{remark}

\section{Proof of Theorem \ref{the:tent}}\label{sec:proof}
\setcounter{equation}{0}

We note that the statement of Theorem \ref{the:tent}
is invariant 
with respect to the natural scaling of the equation
\begin{equation}
  u_{0}(x) \rightarrow u_{0}^{\lambda}(x) 
  := \lambda u_{0}(\lambda x) ,
  \qquad
  u (t,x) \rightarrow u^{\lambda}(t, x) 
  := \lambda u (\lambda^{2} t ,\lambda x).
\end{equation}
Thus it is sufficient to prove the result for 
$u_{0}^{\lambda}(x)$, $u^{\lambda}(t,x)$ instead of
$u_{0}(x)$, $u(t,x)$, for an appropriate choice
of the parameter $\lambda$. We choose 
$\lambda=\overline{\lambda}$ such that
the following two quantities are equal:
\begin{equation}\label{Omega1}
  \Gamma_{1}(\lambda, u_{0}, \widetilde{p}) :=
  \left( \int_{0} ^{+\infty}
  \| u^{\lambda}_{0} (\rho \ \cdot) \|^{\widetilde{p} / 2}_{L^{\widetilde{p}}_{\theta}}  
  \ \rho \ d\rho \right)^{\frac{1}{2}}
  \equiv
  \lambda^{\frac{\widetilde{p}}{4}-1}
  \||x|^{-\frac{2}{\widetilde{p}}}u_{0}\|^{\frac{\widetilde{p}}{4}}
  _{L^{\widetilde{p}/2}_{|x|}L^{\widetilde{p}}_{\theta}}
\end{equation}
\begin{equation}\label{Gamma1}
  \Gamma_{2} (\lambda, u_{0}, \widetilde{p}) :=
  \left( \int_{0} ^{+\infty}
  \| u^{\lambda}_{0} (\rho \ \cdot) \|^{\widetilde{p}}_{L^{\widetilde{p}}_{\theta}}  
  \ \rho \ d\rho \right)^{\frac{1}{2}}
  \equiv
  \lambda^{\frac{\widetilde{p}}{2}-1}
  \||x|^{-\frac{1}{\widetilde{p}}}u_{0}\|^{\frac{\widetilde{p}}{2}}
  _{L^{\widetilde{p}}_{x}}.
\end{equation}
It obvious that such a $\overline{\lambda}$ exists and that
\begin{equation}\label{eq:omga}
  \Gamma_{1}(\overline{\lambda}, u_{0}, \widetilde{p})=
  \Gamma_{2}(\overline{\lambda}, u_{0}, \widetilde{p})=
  [u_{0}]_{\widetilde{p}}
  \equiv \epsilon.
\end{equation}
In the rest of the proof we shall drop the index $\overline{\lambda}$
and write simply $u_{0}:=u_{0}^{\overline{\lambda}}$,
$u:=u^{\overline{\lambda}}$.

We divide the proof into several steps.
Note that in the course of the proof we shall reserve the symbol
$Z$ to denote several universal constants, which do not depend
on $u_{0},u$ and $\widetilde{p}\in[2,4]$, and which may be
different from line to line (and of course the final meaning of
$Z$ will be the maximum of all such constants).

\subsection{Decomposition of the data}
For $s>0$ to be chosen, we write
\begin{equation*}
  u_{0,<s}(x):=u_{0}(x)
  \ \ \text{if $|u_{0}(x)|<s$,}
  \qquad
  u_{0,<s}(x):=0
  \ \ \text{elsewhere}
\end{equation*}
and we decompose the initial data as
\begin{equation*}
  u_{0}=v_{0}+w_{0},
  \qquad
  w_{0}:=\mathbb{P}u_{0,<s},
  \qquad
  v_{0}:=\mathbb{P}(u_{0}-u_{0,<s})
\end{equation*}
which is possible since $u_{0}=\mathbb{P}u_{0}$.
It is clear that $v_{0},w_{0}$ are divergence free.
Moreover one has
\begin{equation}\label{eq:decomp}
  \begin{array}{lcl}
  \| |x|^{-1/2} w_{0}\|_{L^{2}_{|x|}L^{4}_{\theta}} 
  &\leq & Z s^{1-\frac{\widetilde{p}}{4}} ( \int
  \| u_{0} (\rho \ \cdot) \|^{\widetilde{p} / 2}
      _{L^{\widetilde{p}}_{\theta}}  
  \ \rho \ d\rho )^{\frac{1}{2}} 
  =
  Z s^{1-\frac{\widetilde{p}}{4}}
  \epsilon
  \\
  && \\
  \| |x|^{-1/2} v_{0}\|_{L^{2}_{x}}
  &\leq & Z s^{1-\frac{\widetilde{p}}{2}} 
  ( \int
  \| u_{0} (\rho \ \cdot) \|^{\widetilde{p}}_{L^{\p}_{\theta}}  
   \ \rho \ d\rho )^{\frac{1}{2}}
   =
   Z s^{1-\frac{\widetilde{p}}{2}}
   \epsilon
  \end{array}
\end{equation}
for some universal constant $Z\ge1$.
To prove \eqref{eq:decomp}, we use first the fact that
the Leray projection $\mathbb{P}$ is bounded on the weighted 
spaces $L^{2}(\mathbb{R}^{3},|x|^{-1}dx)$ and
$L^{2}_{|x|}L^{4}_{\theta}(\mathbb{R}^{3},|x|^{-1}dx)$
(see Remark \ref{PContinuityRem}),
then the elementary inequalities
\begin{equation*}
  \textstyle
  \||x|^{-1/2}u_{0,<s}\|_{L^{2}_{|x|}L^{4}_{\theta}}
  \le
  s^{1-\frac{\widetilde{p}}{4}} 
  (\int 
    \| u_{0} (\rho \ \cdot) \|^{\widetilde{p} / 2}
    _{L^{\p}_{\theta}}  
    \ \rho \ d\rho )^{\frac{1}{2}},
\end{equation*}
\begin{equation*}
  \textstyle
  \| |x|^{-1/2} u_{0,>s}\|_{L^{2}_{x}}
  \le
  s^{1-\frac{\widetilde{p}}{2}} 
  (\int
  \| u_{0} (\rho \ \cdot) \|^{\widetilde{p} / 2}
  _{L^{\p}_{\theta}}
  \ \rho \ d\rho )^{\frac{1}{2}},
\end{equation*}
and finally property \eqref{eq:omga}.
Now we choose
\begin{equation*}
  s = \frac{2 \p - 4}{4 - \p}
\end{equation*}
and this gives, with $\theta_{1}=\theta_{1}(\widetilde{p})$
and
$\theta_{2}=\theta_{2}(\widetilde{p})$
defined as above,
\begin{equation}\label{eq:decomp2}
  \| |x|^{-1/2} w_{0}\|_{L^{2}_{|x|}L^{4}_{\theta}} 
  \le
  Z \theta_{1}\epsilon,
  \qquad
  \| |x|^{-1/2} v_{0}\|_{L^{2}_{x}}
  \le
  Z \theta_{2}\epsilon.
\end{equation}

\subsection{Decomposition of the weak solution}
Consider the Cauchy problems
\begin{equation}\label{CauchyNSForW}
\left \{
\begin{array}{rcl}
\partial_{t}w + (w \cdot \nabla) w +\nabla P_{w} -\Delta w & = & 0   \\
\nabla \cdot w & = & 0   \\
w(0) & = & w_{0} \\
P_{w} = R \otimes R \ (w \otimes w),  
\end{array}\right.
\end{equation}
and
\begin{equation}\label{CauchyNSForv}
\left \{
\begin{array}{rcl}
\partial_{t}v + (v \cdot \nabla) v + 
(v \cdot \nabla) w + (w \cdot \nabla) v
 +\nabla P_{v} -\Delta v & = & 0   \\
\nabla \cdot v & = & 0  \\
v(0) & = & v_{0} \\
P_{v} = R \otimes R \ (v \otimes v) + 2 R \otimes R \ (v \otimes w). 
\end{array}\right.
\end{equation}
Applying Theorem \ref{WeighKato} (as in \eqref{DefEpsilon4})
and Corollary \ref{Corh:SmallExistenceAlsoInL2},
and recalling
the first inequality in \eqref{eq:decomp2}, we see that
there exist two absolute constants $\varepsilon _{1}$, $C_{1}$
such that problem \eqref{CauchyNS} has 
a unique global smooth solution $w$ provided the data satisfy
\begin{equation*}
  Z \theta_{1} \epsilon < \varepsilon_{1},
\end{equation*}
and in addition the solution $w$ satisfies the estimate
\begin{equation*}
  \| w \|_{L^{8}_{t}L^{4}_{x}}  \leq  C_{1} 
  \| |x|^{-1/2} w_{0} \|_{L^{2}_{|x|}L^{4}_{\theta}} \leq 
  C_{1} Z \theta_{1} \epsilon
  \quad\implies\quad
  \| w \|_{L^{8}_{t}L^{4}_{x}} ^{8}
  \le
  C_{1}^{8}(Z \theta_{1}\epsilon)^{7}\cdot Z \theta_{1}\epsilon.
\end{equation*}
By possibly increasing $Z$ so that it is larger than
both $\varepsilon_{1}^{-1}$ and $C_{1}^{8}$,
this implies the following:
if $\epsilon$ satisfies
\begin{equation}\label{eq:secondZ}
  Z \theta_{1}\epsilon\le1
\end{equation}
then problem \eqref{CauchyNSForW} has a unique global
smooth solution $w$ such that
\begin{equation}\label{BoundOnW}
  \| w \|_{L^{8}_{t}L^{4}_{x}}^{8} \le Z \theta_{1} \epsilon.
\end{equation}
As a consequence, the function $v=u-w$ is a 
weak solution of the second Cauchy problem 
\eqref{CauchyNSForv}\footnote{Notice that $v \rightharpoonup v_{0}$ in $L^{2}$ because $u \rightharpoonup u_{0}$ in $L^{2}$ (being $u$ a suitable weak solution of (\ref{CauchyNS})) and $w \rightarrow w_{0}$ in $L^{2}$ (by Corollary (\ref{Corh:SmallExistenceAlsoInL2})).}. Moreover, since $u$ is a suitable
weak solution, the function $v$ inherits similar properties
(we shall say for short that $v$ is a suitable weak solution
of the modified problem \eqref{CauchyNSForv}).

\subsection{A change of variables}

Let $\xi \in \mathbb{R}^{3}\setminus0$, $T > 0$ and consider
the segment 
$$
L(T,\xi) := \{ (s, \xi s) \ : \ s \in (0,T) \}.
$$
We ask for which $(T,\xi)$ the set $L(T, \xi)$ is a regular set. 
To this purpose we introduce the transformation
\begin{equation}\label{TransfI}
  (t, y) = (t, x - \xi t),
  \qquad
  v_{\xi}(t,y) = v(t,x),
  \qquad
  w_{\xi}(t,y) = w(t,x),
\end{equation}
which takes \eqref{CauchyNSForW} into the system
\begin{equation}\label{CauchyNSForwxi}
\left \{
\begin{array}{rcl}
\partial_{t}w_{\xi} + 
((w_{\xi} -\xi)\cdot \nabla) w_{\xi} 
+\nabla P_{w_{\xi}} -\Delta w_{\xi} & = & 0   \\
\nabla \cdot w_{\xi} & = & 0   \\
w_{\xi}(0) & = & w_{0} \\
P_{w_{\xi}} = R \otimes R \ (w_{\xi} \otimes w_{\xi}),  
\end{array}\right.
\end{equation}
and (\ref{CauchyNSForv}) into the system
\begin{equation}\label{CauchyNSForvXi}
  \left \{
  \begin{array}{rcl}
  \partial_{t}v_{\xi} + ((v_{\xi} - \xi )\cdot \nabla) v_{\xi} + 
  (v_{\xi} \cdot \nabla) w_{\xi} + (w_{\xi} \cdot \nabla) v_{\xi}
   +\nabla P_{v_{\xi}} -\Delta v_{\xi} & = & 0   \\
  \nabla \cdot v_{\xi} & = & 0  \\
  v_{\xi}(0) & = & v_{0} \\
  P_{v_{\xi}} =  R \otimes R \ (v_{\xi} \otimes v_{\xi}) 
  + 2 R \otimes R \ (v_{\xi} \otimes w_{\xi}). 
  \end{array}\right.
\end{equation}
Note that this change of coordinates maps $L(T, \xi)$ 
in $(0,T) \times \left\{ 0 \right\} $.
Now we fix an arbitrary $M\ge1$ and we define the set
\begin{equation}\label{DefSetOfBarS}
  S(M, T, \xi) := \left\{ s \in [0,T] \ : \
  \int_{s}^{s + T/M} \int_{\RT} |y|^{-1}
  |\nabla v_{\xi} (\tau , y)|^{2} \ d\tau dy > M
  \right\}
\end{equation}
and the number $\overline{s}\ge0$
\begin{equation}\label{DefOfBarS}
\bar{s} := 
\begin{cases}
  \inf \left\{ s \in S(M, T, \xi) \right\}   
  & \mbox{if} \quad S(M, T, \xi) \neq \emptyset    \\
  T   &   \mbox{otherwise}.
\end{cases}
\end{equation}
From the definition of $\bar{s}$ one has immediately
\begin{equation}\label{BasicProperties}
  \int_{0}^{\bar{s}} \int_{\RT} |y|^{-1}
  |\nabla v_{\xi}(\tau ,y)|^{2}  \ d \tau dy  \leq M(M+1)
  \le 
  2M^{2}.
\end{equation}
We next distinguish two cases.

\subsection{First case: \texorpdfstring{$\bar{s} = T$}{}}
In this case the entire segment $L(T,\xi)$ is a regular 
set. To prove this, we note first that by
\eqref{BasicProperties}
\begin{equation}\label{FinalBound1}
  \int_{0}^{T} \int_{\mathbb{R}^{3}}
  \frac{| \nabla v
  (\tau , x)|^{2} }{|x - \xi \tau |} \ d \tau dx 
  < + \infty.
\end{equation}
Suppose we can also prove that
\begin{equation}\label{FinalBound2}
  \int_{0}^{T} \int_{\mathbb{R}^{3}}
  \frac{| \nabla w (\tau ,x)|^{2} }{|x - \xi \tau |} 
  \ d \tau dx  < + \infty
\end{equation}
Then summing the two we obtain
\begin{equation}
  \int_{0}^{T} \int_{\mathbb{R}^{3}}
  \frac{| \nabla u (\tau ,x)|^{2} }{|x - \xi \tau |} 
  \ d\tau dx  < + \infty.
\end{equation}
Now let $0 < s < T $, and let $r > 0$ be so small 
that $0 < s - 7r^{2}/8  < s + r^{2}/8  < T$
and $|\xi| r \leq 1$. For each $(\tau , x) \in Q^{*}_{r}(s, \xi s)$
we have 
\begin{equation*}
  |x - \xi \tau | \le
  |x-\xi s|+|\xi||s-\tau|\le
  r + r^{2} |\xi| \leq 2r
\end{equation*}
which implies
\begin{equation*}
  \frac{1}{r} \int \int_{Q^{*}_{r}(s, \xi s)} 
  |\nabla u(\tau , x)|^{2} \  d \tau d x 
  \le
  2 \int_{s - \frac{7}{8}r^{2}}^{s 
  + \frac{1}{8}r^{2}} \int_{\mathbb{R}^{3}}
  \frac{|\nabla u(\tau, x)|^{2}}{|x - \xi \tau|} \  d \tau d x .  
\end{equation*}
By continuity of the integral function, we obtain that
the regularity condition (\ref{CKNConditionSB}) is satisfied
at all $(s,\xi s)\in L(T,\xi)$, {\it i.e.}
$L(T,\xi)$ is a regular set as claimed.

It remains to prove \eqref{FinalBound2}.
By \eqref{BoundOnW}, \eqref{eq:decomp2}
we know that
\begin{equation}\label{HPonW}
  \|  w_{\xi}  \|_{L^{8}_{t}L^{4}_{x}}=
  \|  w \|_{L^{8}_{t}L^{4}_{x}} < + \infty, 
  \qquad  \| |x|^{- 1/2} w_{0} \|_{L^{2}(\mathbb{R}^{3})} 
  < + \infty
\end{equation}
and that $w$, hence $w_{\xi}$, is a smooth solution.
Thus we can write the energy inequality
\begin{equation}\label{eq:energyw}
  \begin{split}
    \textstyle
    \int_{\mathbb{R}^{3}} \phi 
    &
    \textstyle
    |w_{\xi}|^{2} dx + 
    2 \int_{0}^{t}\int_{\mathbb{R}^{3}} \phi |\nabla  w_{\xi}|^{2} 
    \le
    \\
    & \le
    \textstyle
    \int_{\mathbb{R}^{3}} \phi |w_{0}|^{2}dx + 
    \int_{0}^{t} \int_{\mathbb{R}^{3}} 
    |w_{\xi}|^{2} (\phi_{t} - \xi \cdot \nabla \phi + \Delta \phi)
    \int_{0}^{t} \int_{\mathbb{R}^{3}} (|w_{\xi}|^{2} + 
    2P_{w_{\xi}})w_{\xi} \cdot \nabla \phi
  \end{split}
\end{equation}
where $P_{w_{\xi}} = R \otimes R \ (w_{\xi} \otimes w_{\xi})$
and $\phi\in C^{\infty}_{c}(\mathbb{R}^{3})$ is any test function
$\phi\ge0$. We choose
\begin{equation*}
  \textstyle
  \phi(y) := \sigma_{\nu} (y) \chi(\delta |y|),
  \qquad \sigma_{\nu}(y) := (\nu+|y|^{2})^{-\frac12}
  \qquad
  \nu , \delta > 0
\end{equation*}
where $\chi$ is a cut-off function supported in $[-1,1]$ and equal
to 1 near 0 (compare with the proof of Lemma 8.2 in \cite{CKN}).
Letting $\delta \rightarrow 0$
and using the inequalities
\begin{equation}\label{eq:sigmapr}
  \textstyle
 |\nabla \sigma_{\nu}| \leq (\nu + |y|^{2})^{-1} 
 = \sigma_{\nu}^{2},
 \qquad \Delta \sigma_{\nu} < 0,
\end{equation}
we obtain
\begin{equation*}
  \textstyle
  \left[ \int_{\mathbb{R}^{3}} \sigma_{\nu} 
  |w_{\xi}|^{2} \right]_{0}^{t} + 
  2\int_{0}^{t}\int_{\mathbb{R}^{3}} 
  \sigma_{\nu} |\nabla w_{\xi}|^{2} 
  \le
  |\xi|\int_{0}^{t}\int_{\mathbb{R}^{3}}  
  \sigma_{\nu}^{2} |w_{\xi}|^{2}  
   +  \int_{0}^{t}\int_{\mathbb{R}^{3}} \sigma_{\nu}^{2}
   (|w_{\xi}|^{3} + 2|P_{w_{\xi}}||w_{\xi}| ).
\end{equation*}
Our goal is to prove an integral inequality
for the quanities
\begin{equation}\nonumber
a_{\nu}(t) = \int_{\mathbb{R}^{3}} \sigma_{\nu} |w_{\xi}|^{2}(t), 
\qquad
B_{\nu}(t) = \int_{0}^{t}\int_{\mathbb{R}^{3}} 
\sigma_{\nu} | \nabla w_{\xi}|^{2} .
\end{equation}
To proceed, we use the weighted $L^{p}$ inequality
for the Riesz transform (\cite{Stein}), uniform in
$\nu\ge0$
\begin{equation}\label{SteinIneq}
  \|\sigma_{\nu}^{m} R \phi \|_{L^{s}} 
  \leq Z \|\sigma_{\nu}^{m}\phi\|_{L^{s}}, 
  \qquad
  \textstyle
  1 < s < \infty,
  \qquad
  m \in 
    \left( - \frac{3(s-1)}{s}, \frac{3}{s} \right).
\end{equation}
For the pressure term we have, using 
\eqref{CKNInequalityCKNVersion} and \eqref{SteinIneq},
\begin{eqnarray}\nonumber
  \textstyle
  2 \int_{\mathbb{R}^{3}} \sigma_{\nu}^{2}
  |P_{w_{\xi}}||w_{\xi}| 
  & = & 
  \textstyle
  2 \int_{\mathbb{R}^{3}} \sigma_{\nu}^{2}
  |w_{\xi}| | R \otimes R \  
  ( w_{\xi} \otimes w_{\xi} ) |
  \\ \nonumber
  & \leq  & 
  \| \sigma_{\nu}
  R \otimes R \  
  ( w_{\xi} \otimes w_{\xi} )   \|_{L^{8/5}}
  \| \sigma_{\nu}
  w_{\xi} \|_{L^{8/3}}
  \\ \nonumber
  & \lesssim &
  \| \sigma_{\nu}
  | w_{\xi} |^{2}   \|_{L^{8/5}}
  \| \sigma_{\nu}
  w_{\xi} \|_{L^{8/3}}
  \\ \nonumber
  & \leq &
  \|  w_{\xi}    \|_{L^{4}}
  \| \sigma_{\nu}
  w_{\xi} \|_{L^{8/3}}^{2}
  \\ \nonumber
  & \lesssim &
  \|  w_{\xi}    \|_{L^{4}}
  \| \sigma_{\nu}^{1/2}
  \nabla w_{\xi} \|_{L^{2}}^{7/4}
  \| \sigma_{\nu}^{1/2}
  w_{\xi} \|_{L^{2}}^{1/4}
  \\ \label{Appendx:FirstIn}
  & = &
  \|  w_{\xi}    \|_{L^{4}}
  \dot{B}^{7/4}_{\nu}a_{\nu}^{1/4}
  \le
  \textstyle
  \frac{\dot{B}_{\nu}}{6}  
   +  C \| w_{\xi} \|_{L^{4}}^{8} \cdot a_{\nu}
\end{eqnarray} 
for some universal constant $C$.
In a similar way,
\begin{eqnarray}
  \textstyle
  \int_{\mathbb{R}^{3}} \sigma_{\nu}^{2}
  |w_{\xi}|^{3}
  & \leq &
  \| w_{\xi} \|_{L^{4}}
  \| \sigma_{\nu}^{2}  |w_{\xi}|^{2} \|_{L^{4/3}}
  \\ \nonumber
  &  =  &
  \| w_{\xi} \|_{L^{4}}
  \| \sigma_{\nu}  w_{\xi}  \|_{L^{8/3}}^{2}
  \\ \nonumber
  & \lesssim &
  \|  w_{\xi}    \|_{L^{4}}
  \| \sigma_{\nu}^{1/2}
  \nabla w_{\xi} \|_{L^{2}}^{7/4}
  \| \sigma_{\nu}^{1/2}
  w_{\xi} \|_{L^{2}}^{1/4}
  \\ \label{Appendx:SecondIn}
  &  \le  &
  \textstyle
  \frac{\dot{B}_{\nu}}{6}  
  +  C \| w_{\xi} \|_{L^{4}}^{8}  a_{\nu}
\end{eqnarray}
and
\begin{equation}\label{Appendx:ThirdIn}
  \textstyle
  |\xi| \int_{\RT} \sigma_{\nu}^{2} |w_{\xi}|^{2} 
  \lesssim 
  |\xi|\cdot  \| \sigma_{\nu}^{1/2} 
  \nabla w_{\xi}  \|_{L^{2}}
  \| \sigma_{\nu}^{1/2} w_{\xi} \|_{L^{2}} 
  =
  |\xi|   (\dot{B}_{\nu} a_{\nu})^{1/2} 
  \leq \frac{\dot{B}_{\nu}}{6} + C |\xi|^{2} a_{\nu} .
\end{equation}
Plugging these inequalities in the energy estimate
we get
\begin{equation}\label{Appendix:FinalINeq}
  a_{\nu}(t) + B_{\nu}(t) \leq 
  a_{\nu}(0) +
  \int_{0}^{t}  \left( 
  C |\xi|^{2} + 
  3 C \|  w_{\xi}(s,\cdot)  \|^{8}_{L^{4}}
  \right) a(s) \ ds,
\end{equation}
and passing to the limit $\nu \rightarrow 0$
we obtain, for some larger universal constant $C$
(note that $\|w_{\xi} (t) \|_{L^{4}} = \| w (t) \|_{L^{4}}$ 
for all $t$)
\begin{equation*}
  a(t) + B(t) \leq 
  a(0)  +
  C\int_{0}^{t}  \left( 
  |\xi|^{2} + 
   \|  w (s,\cdot) \|^{8}_{L^{4}}
  \right) a(s) \ ds,
\end{equation*}
where
\begin{equation*}
  a(t) = \int_{\mathbb{R}^{3}} |y|^{-1} |w_{\xi}|^{2}(t) , 
  \qquad
  B(t) = \int_{0}^{t} \int_{\mathbb{R}^{3}} |y|^{-1} 
  |\nabla w_{\xi}|^{2}.
\end{equation*}
By a standard application of Gronwall's inequality,
we obtain $a(t) < + \infty$ 
for all $t \geq 0$ which implies also
$B(t) < + \infty$ for all $t \geq 0$ as claimed.

\subsection{Second case: \texorpdfstring{$0 \leq \bar{s} < T$}{} }
Since $v_{\xi}$ is a suitable weak solution of 
Problem \eqref{CauchyNSForvXi}, the following
generalized energy inequality is valid
(see e.g.~\cite{Calderon} for details):
for all $t\ge0$ and
$\phi\in C^{\infty}_{c}(\mathbb{R^{+}}\times \mathbb{R}^{3})$,
we have
\begin{equation*}
\begin{split}
  \textstyle
  \int_{\mathbb{R}^{3}} \phi(t,x) 
  |v_{\xi}|^{2} dx
  &+ 
  \textstyle
  2 \int_{0}^{t}\int_{\mathbb{R}^{3}} 
  \phi |\nabla  v_{\xi}|^{2} 
  \le 
  \\
  \le&
  \textstyle
  \int_{\mathbb{R}^{3}} \phi(0,x) |v_{0}|^{2}dx + 
  \int_{0}^{t} \int_{\mathbb{R}^{3}} 
  |v_{\xi}|^{2} (\phi_{t} - \xi \cdot \nabla \phi + \Delta \phi)
  +
  \\
  &+
  \textstyle
  \int_{0}^{t} \int_{\mathbb{R}^{3}} (|v_{\xi}|^{2} 
  + 2P_{v_{\xi}})v_{\xi} \cdot \nabla \phi
  +
  \int_{0}^{t}\int_{\mathbb{R}^{3}} 
  |v_{\xi}|^{2} (w_{\xi} \cdot \nabla \phi) 
  \\
  &+
  \textstyle
  2 \int_{0}^{t}\int_{\mathbb{R}^{3}} 
  (v_{\xi} \cdot w_{\xi}) (v_{\xi} \cdot \nabla \phi)  
  + \phi (v_{\xi} \cdot \nabla ) v_{\xi} \cdot w_{\xi}  
\end{split}
\end{equation*}
which implies
\begin{equation}\label{GenEn} 
  \begin{split}
    \textstyle
    \int_{\mathbb{R}^{3}} \phi(t,x) 
    &|v_{\xi}|^{2} dx
    + 
    \textstyle
    2 \int_{0}^{t}\int_{\mathbb{R}^{3}} 
    \phi |\nabla  v_{\xi}|^{2} 
    \le 
    \\
    \le&
    \textstyle
    \int_{\mathbb{R}^{3}} \phi(0,x) |v_{0}|^{2} dx+ 
    \int_{0}^{t} \int_{\mathbb{R}^{3}} 
    |v_{\xi}|^{2} (\phi_{t} - \xi \cdot \nabla \phi + \Delta \phi)
    +
    \\
    &+
    \textstyle
    \int_{0}^{t} \int_{\mathbb{R}^{3}} (|v_{\xi}|^{2} 
    + 2P_{v_{\xi}})v_{\xi} \cdot \nabla \phi
    +
    \int_{0}^{t}\int_{\mathbb{R}^{3}} 3 |v_{\xi}|^{2} 
    |w_{\xi}| |\nabla \phi|
    + 18 | \phi | |v_{\xi}| |\nabla v_{\xi}| |w_{\xi}|.
  \end{split}
\end{equation}
By a standard approximation procedure
(see the proof of Lemma 8.3 in \cite{CKN}) 
the estimate is valid for any test function of the form
\begin{equation*}
  \phi(t,y) := \psi(t) \phi_{1}(y)
\end{equation*}
with
$
\phi_{1} \in C^{\infty}_{c}(\mathbb{R}^{3})$, 
$\phi_{1} \geq 0,$
and 
\begin{equation*}
  \psi : \Rpiu \rightarrow \mathbb{R}
  \quad \mbox{absolutely continuous with} 
  \quad  \dot{\psi} \in L^{1}(\Rpiu).
\end{equation*}
We shall choose here
\begin{equation*}
  \psi(t)\equiv1,
  \qquad
  \phi_{1} =\sigma_{\nu}(y) \chi(\delta |y|), 
\end{equation*}
where $\nu,\delta>0$,
\begin{equation*}
  \sigma_{\nu}(y) = (\nu + |y|^{2})^{-\frac12},
\end{equation*}
and $\chi:\mathbb{R}^{+}\to \mathbb{R}^{+}$
is a smooth nonincreasing function such that
\begin{equation*}
  \chi =1 \ \text{on}\  [0,1],
  \qquad
  \chi =0  \ \text{on}\ [2, +\infty].
\end{equation*} 
Passing to the limit $\delta \rightarrow 0$ in the 
energy inequality we obtain
\begin{equation}\label{PassToTheLim}
\begin{split}
  \textstyle
  \left[ \int_{\mathbb{R}^{3}} \sigma_{\nu} 
    |v_{\xi}|^{2} \right]_{0}^{t} 
  &+
  \textstyle
  2\int_{0}^{t}\int_{\mathbb{R}^{3}} \sigma_{\nu} 
    |\nabla v_{\xi}|^{2}\le
  \\
  \le &
  \textstyle
  \int_{0}^{t}\int_{\mathbb{R}^{3}}  
  |v_{\xi}|^{2}(- \xi \cdot \nabla \sigma_{\nu} 
  + \Delta \sigma_{\nu}) 
  +
  \textstyle
  \int_{0}^{t}\int_{\mathbb{R}^{3}} (|v_{\xi}|^{2} 
  + 2P_{v_{\xi}})v_{\xi}\cdot \nabla \sigma_{\nu}
  \\
  &+
  \textstyle
  18 \int_{0}^{t}\int_{\mathbb{R}^{3}} 
  \sigma_{\nu}  |v_{\xi}| |\nabla v_{\xi}| |w_{\xi}| 
  + 
  3 \int_{0}^{t}\int_{\mathbb{R}^{3}} 
  |v_{\xi}|^{2} | w_{\xi} | | \nabla \sigma_{\nu} |.
\end{split}
\end{equation}
Note that a similar argument is used in \cite{CKN},
one of the differences here being the presence of the last two 
perturbative terms, which
we control using \eqref{BoundOnW}. 
Recalling \eqref{eq:sigmapr}, we deduce the estimate
\begin{equation}\label{PutIn}
\begin{split}
  \textstyle
  \left[ \int_{\mathbb{R}^{3}} \sigma_{\nu} 
    |v_{\xi}|^{2} \right]_{0}^{t} 
  &+
  \textstyle
  2\int_{0}^{t}\int_{\mathbb{R}^{3}} \sigma_{\nu} 
    |\nabla v_{\xi}|^{2}\le
  \\
  \le
  |\xi|
  &
  \textstyle
  \int_{0}^{t}\int_{\mathbb{R}^{3}} 
    \sigma_{\nu}^{2} |v_{\xi}|^{2}
  +\int_{0}^{t}\int_{\mathbb{R}^{3}} \sigma_{\nu}^{2}
    (|v_{\xi}|^{3} + 2|P_{v_{\xi}}||v_{\xi}|
    +3 |v_{\xi}|^{2}|w_{\xi}|) 
    + 18 \sigma_{\nu} |v_{\xi}| | \nabla v_{\xi} | 
    | w_{\xi} |.
\end{split}
\end{equation}
We can now proceed as in the first case, using
\eqref{PutIn} to obtain a Gronwall type
inequality for the quantities
\begin{equation}\nonumber
  a_{\nu}(t) = 
  \|\sigma_{\nu}^{1/2}v_{\xi}\|_{L^{2}}^{2},
  \qquad
  B_{\nu}(t) = 
  \int_{0}^{t}\|\sigma_{\nu}^{1/2}
    \nabla v_{\xi}(s)\|_{L^{2}}^{2}ds.
\end{equation}
We first estimate the term in $P_{v_{\xi}}$; recall that
\begin{equation*}
  P_{v_{\xi}} = R \otimes R \ (v_{\xi} \otimes v_{\xi}) 
  + 2 R \otimes R \ (v_{\xi} \otimes w_{\xi}).
\end{equation*}
We have
\begin{equation*}
  \textstyle
  2\int_{\RT} \sigma_{\nu}^{2} |P_{v_{\xi}}| |v_{\xi}|
  \leq 
  2\int_{\RT} \sigma_{\nu}^{2} |v_{\xi}| 
  |R \otimes R \ (v_{\xi} \otimes v_{\xi})| 
  + 4 \int_{\RT} \sigma_{\nu}^{2} |v_{\xi}| 
  |R \otimes R \ (v_{\xi} \otimes w)|
   =: I + II.
\end{equation*}
Here and in the following, as usual, $Z$ denotes several
universal constats, possibly different from line to line.
By \eqref{SteinIneq} we can write
\begin{equation*}
  I\le
  2 \|\sigma_{\nu} 
  R \otimes R  \ (v_{\xi} \otimes v_{\xi})\|_{L^{2}}
  \le
  Z
  \|\sigma_{\nu} |v_{\xi}|^{2} \|_{L^{2}}
  \| \sigma_{\nu}  v_{\xi}  \|_{L^{2}} 
  \le
  Z
  \|\sigma_{\nu}^{1/2} v_{\xi} \|_{L^{4}}^{2}
  \| \sigma_{\nu}  v_{\xi}  \|_{L^{2}} 
\end{equation*}
and then by the Caffarelli--Kohn--Nirenberg inequality we obtain
\begin{equation}\label{FirstIn}
  I\le
  Z
  \|\sigma_{\nu}^{1/2} \nabla v_{\xi}  \|^{3/2}_{L^{2}}
  \| \sigma_{\nu}^{1/2}  v_{\xi}  \|_{L^{2}}^{1/2} 
  \cdot
  \|\sigma_{\nu}^{1/2} \nabla v_{\xi}  \|^{1/2}_{L^{2}}
  \| \sigma_{\nu}^{1/2}  v_{\xi}  \|_{L^{2}}^{1/2} 
  =
  Z
  \dot{B}_{\nu} a_{\nu}^{1/2} 
  \le 
  \frac{\dot{B}_{\nu}}{6}  
  + Z \dot{B}_{\nu}a_{\nu}.
\end{equation}
In a similar way we have
\begin{equation*}
  II\le
  4 \| \sigma_{\nu} 
  R \otimes R  \ (v_{\xi} \otimes w) \|_{L^{\frac85}}
  \| \sigma_{\nu}  v_{\xi} \|_{L^{\frac83}} 
  \le
  Z
  \| \sigma_{\nu} |v_{\xi}| |w| \|_{L^{\frac85}}
  \| \sigma_{\nu}  v_{\xi}  \|_{L^{\frac83}}
  \le
  Z
   \|  w \|_{L^{4}}
  \| \sigma_{\nu}  v_{\xi}  \|_{L^{\frac83}}^{2}
\end{equation*}
and again by the CKN inequality
\begin{equation}\label{SecIn}
  II
  \le
  Z
  \|  w \|_{L^{4}}
  \| \sigma_{\nu}^{1/2}  v_{\xi}  \|^{1/4}_{L^{2}} 
  \| \sigma_{\nu}^{1/2}  \nabla v_{\xi}  \|^{7/4}_{L^{2}},
  =
  Z
  \|  w \|_{L^{4}}
  a_{\nu}^{\frac18}\dot{B_{\nu}}^{\frac78}
  \le
  \frac{\dot{B}_{\nu}}{6} 
  + Z \| w \|^{8}_{L^{4}} a_{\nu}. 
\end{equation}
Consider now the other terms in \eqref{PutIn}.
Proceeding as above, we have
\begin{equation}\label{ThirdIn}
  |\xi|
  \int_{\RT} \sigma_{\nu}^{2} |v_{\xi}|^{2} 
  \le
  Z
  |\xi|
  \| \sigma_{\nu}^{1/2} \nabla v_{\xi}  \|_{L^{2}}
  \| \sigma_{\nu}^{1/2} v_{\xi} \|_{L^{2}} =
  Z|\xi|(\dot{B}_{\nu}a_{\nu})^{1/2}
  \le
  \frac{\dot{B}_{\nu}}{6} + Z |\xi|^{2} a_{\nu};
\end{equation}
and
\begin{equation}\label{FourthIn}
  \int_{\RT} \sigma_{\nu}^{2} |v_{\xi}|^{3} 
  = \| \sigma_{\nu}^{2/3} v_{\xi} \|_{L^{3}}^{3} 
  \le
  Z
  \| \sigma_{\nu}^{1/2} \nabla v_{\xi}  \|^{2}_{L^{2}}
  \| \sigma_{\nu}^{1/2} v_{\xi} \|_{L^{2}}
  =
  Z
  \dot{B}_{\nu} a_{\nu}^{1/2} 
  \leq 
  \frac{\dot{B}_{\nu}}{6}  + Z \dot{B}_{\nu}a_{\nu}
\end{equation}
while for the perturbative terms we can write
\begin{equation}\label{SixthIn}
\begin{split}
  3 \int_{\RT}\sigma_{\nu}^{2} |v_{\xi}|^{2} |w_{\xi}|
  \le&
  3 \|  w_{\xi} \|_{L^{4}}
  \| \sigma_{\nu}  v_{\xi}  \|_{L^{8/3}}^{2}
  \le
  Z
  \|  w_{\xi} \|_{L^{4}}
  \| \sigma_{\nu}^{1/2}  v_{\xi}  \|^{1/4}_{L^{2}} 
  \| \sigma_{\nu}^{1/2}  \nabla v_{\xi}  \|^{7/4}_{L^{2}}
  \\
  =&
  Z\|  w_{\xi} \|_{L^{4}}
  a_{\nu}^{1/8}
  \dot{B}_{\nu}^{7/8}
  \le
   \frac{\dot{B}_{\nu}}{6}  
  + Z \|  w_{\xi} \|^{8}_{L^{4}} a_{\nu}
\end{split}
\end{equation}
and
\begin{equation}\label{SevIn}
\begin{split}
  18 \int_{\RT}\sigma_{\nu} |v_{\xi}| |\nabla v_{\xi}| 
  |w_{\xi}| 
  \le&
  18  \| \sigma_{\nu}^{1/2}  \nabla v_{\xi}  \|_{L^{2}}
  \| w_{\xi} \|_{L^{4}} 
  \| \sigma_{\nu}^{1/2}  v_{\xi}  \|_{L^{4}}
  \\
  \le&
  Z
  \| \sigma_{\nu}^{1/2}  \nabla v_{\xi}  \|_{L^{2}}
  \| w_{\xi} \|_{L^{4}}
  \| \sigma_{\nu}^{1/2}  \nabla v_{\xi}  \|^{3/4}_{L^{2}}
  \| \sigma_{\nu}^{1/2}  v_{\xi}  \|^{1/4}_{L^{2}}
  \\
  =&
  Z\| w_{\xi} \|_{L^{4}}
  \dot{B}_{\nu}^{7/8}a_{\nu}^{1/8}
  \le
   \frac{\dot{B}_{\nu}}{6}  
  + Z \|  w_{\xi} \|^{8}_{L^{4}} a_{\nu}.
\end{split}
\end{equation}
Now recalling \eqref{PutIn},
summing all the inequalities and absorbing
a term 
$  \int_{0}^{t}\dot{B}_{\nu}(s)ds\equiv 
  B_{\nu}(t)$ from the left hand side, we obtain
\begin{equation*}
  a_{\nu}(t) + B_{\nu}(t) \leq 
  a_{\nu}(0) +
  Z\int_{0}^{t}  \left( 
  |\xi|^{2} + 
  \dot{B}_{\nu}(s) +  
  \|  w_{\xi} (s,\cdot) \|^{8}_{L^{4}}
  \right) a(s) \ ds,
\end{equation*}
and passing to the limit $\nu \rightarrow 0$,
we arrive at the estimate
\begin{equation*}
  a(t) + B(t) \leq 
  a(0)  +
  Z\int_{0}^{t}  \left( 
  |\xi|^{2} + 
  \dot{B}(s) + 
  \|  w_{\xi}(s,\cdot)  \|^{8}_{L^{4}}
  \right) a(s) \ ds,
\end{equation*}
for some universal constant $Z$, where
\begin{equation*}
  a(t) = \int_{\mathbb{R}^{3}} |y|^{-1} 
    |v_{\xi}(t,y)|^{2}dy, 
  \qquad
  B(t) = \int_{0}^{t} \int_{\mathbb{R}^{3}} 
  |y|^{-1} |\nabla v_{\xi}(s,y)|^{2}dsdy.
\end{equation*}
By a standard application of Gronwall's lemma we get
for $0\le t\le \bar{s}$
\begin{equation*}
  a(t)\le a(0)(1+ZA e^{ZA}),
  \qquad
  A=B(\bar{s})+\|w\|_{L^{8}_{t}L^{4}_{x}}^{8}+\bar{s}|\xi|^{2}.
\end{equation*}
By \eqref{BasicProperties}, \eqref{BoundOnW} we have
$A\le 2M^{2}+Z+\bar{s}|\xi|^{2}$, while by
\eqref{eq:decomp2} we have $a(0)\le (Z \theta_{2}\epsilon)^{2}$
(note that $w_{\xi}$, $v_{\xi}$ at fixed $t$ are just translations
of $w,v$ respectively).
If we restrict to the 
vectors $\xi$ such that\footnote{Remember 
that $\bar{s}$ is a function of $\xi$.} 
\begin{equation}\label{ResrtCondFirst}
|\xi|^{2}\bar{s} \leq M^{2}
\end{equation} 
the estimate becomes
\begin{equation*}
  a(\bar{s})\le(Z\theta_{2}\epsilon)^{2}
    (1+(3M^{2}+Z)e^{3M^{2}+Z})
\end{equation*}
and taking a possibly larger universal constant $Z$, this implies
\begin{equation}\label{BoundBySmallnessBis}
  a(\bar{s})\le Ze^{4M^{2}}(\theta_{2}\epsilon)^{2}.
\end{equation}
Notice that (\ref{ResrtCondFirst}) is satisfied 
provided that
\begin{equation}\label{RestrFirstOnL}
L(T , \xi) \subset \left\{ (\tau , z): \tau \geq \frac{|z|^{2}}{M^{2}} \right\}.
\end{equation}

We now repeat the argument, starting from the point
$(\bar{s},\bar{s}\xi)$. We write the analogous of the energy
inequality (\ref{GenEn}) on
the time interval 
$\bar{s}\le s\le t$ with $t\le \bar{s}+ T$,
choosing as test function 
$\phi(t,y):=\psi_{\nu}(t)\sigma_{\nu}(y) \chi(\delta |y|)$
where $\chi$ and $\sigma_{\nu}$ are as before, while
\begin{equation*}
  \psi_{\nu}(t) := e^{-k B_{\bar{s}, \nu}(t)},
  \qquad 
  B_{\bar{s}, \nu}(t) := 
  \int_{\bar{s}}^{t}
  \int_{\mathbb{R}^{3}} 
  \sigma_{\nu} |\nabla v_{\xi}|^{2}
\end{equation*}
with $k$ a positive constant to be specified. 
Note that $B_{\bar{s}, \nu}$ is bounded if
$\nu>0$ by the properties of $v$.
In this way we obtain, letting $\delta\to0$,
\begin{equation*}
\begin{split}
  \textstyle
  [ \int_{\mathbb{R}^{3}} 
  &
  \psi_{\nu} \sigma_{\nu} 
  |v_{\xi}|^{2} ]_{\bar{s}}^{ t}
   + 
  \textstyle
  2\int_{\bar{s}}^{t}\int_{\mathbb{R}^{3}} 
  \psi_{\nu} \sigma_{\nu} |\nabla v_{\xi}|^{2} 
  \le
  \\
  \le &
  \textstyle
  \int_{\bar{s}}^{t}\int_{\mathbb{R}^{3}} 
  \psi_{\nu} |v_{\xi}|^{2}
  (-k \dot{B}_{\bar{s} , \nu} \sigma_{\nu} 
  -\xi \cdot \nabla \sigma_{\nu} + \Delta \sigma_{\nu}) 
  +
  \int_{\bar{s}}^{ t}\int_{\mathbb{R}^{3}} 
  \psi_{\nu} (|v_{\xi}|^{2} 
  +2P_{v_{\xi}}v_{\xi})\cdot \nabla \sigma_{\nu} 
  \\
  +&
  \textstyle
  18 \int_{\bar{s}}^{ t}\int_{\mathbb{R}^{3}} 
  \psi_{\nu} \sigma_{\nu}  |v_{\xi}| |\nabla v_{\xi}| 
  |w_{\xi}| 
  +
  3 \int_{\bar{s}}^{t}\int_{\mathbb{R}^{3}} 
  \psi_{\nu} |v_{\xi}|^{2} | w_{\xi} | 
  | \nabla \sigma_{\nu} |,
\end{split}
\end{equation*}
for $\bar{s}\le t\le \bar{s}+T$,
and this implies, recalling \eqref{eq:sigmapr},
\begin{equation}\label{PutIn2}
\begin{split}
  \textstyle
  [ \int_{\mathbb{R}^{3}} 
  \psi_{\nu} \sigma_{\nu} 
  &
  \textstyle
  |v_{\xi}|^{2} 
  ]_{\bar{s}}^{ t} 
  + 2\int_{\bar{s}}^{ t}\int_{\mathbb{R}^{3}} 
  \psi_{\nu} \sigma_{\nu} |\nabla v_{\xi}|^{2} 
  \le
  \\
  \le &
  \textstyle
  \int_{\bar{s}}^{ t}
  \int_{\mathbb{R}^{3}} \psi_{\nu} 
  |v_{\xi}|^{2}(|\xi| \sigma_{\nu}^{2} 
  - k \dot{B}_{\bar{s} , \nu} \sigma_{\nu})
  \\
  &+
  \textstyle
  \int_{\bar{s}}^{ t}
  \psi_{\nu}
  \int_{\mathbb{R}^{3}} 
  \sigma_{\nu}^{2}
  (|v_{\xi}|^{3} + 2|P_{v_{\xi}}||v_{\xi}|
  +3 |v_{\xi}|^{2}|w_{\xi}|) 
  + 18 \sigma_{\nu} |v_{\xi}| | \nabla v_{\xi} | | w_{\xi} |. 
\end{split}
\end{equation}
Our goal now is to prove an integral inequality 
involving the quantities 
\begin{equation*}
  a_{\nu}( t) = 
  \int_{\mathbb{R}^{3}} \sigma_{\nu} 
  |v_{\xi} ( t,x)|^{2} dx,
  \qquad
  B_{\bar{s}, \nu}(t) = 
  \int_{\bar{s}}^{t} \int_{\mathbb{R}^{3}} 
  \sigma_{\nu} |\nabla v_{\xi}|^{2}.
\end{equation*}
We estimate the terms at the right hand side
of \eqref{PutIn2}. First of all we have
\begin{equation*}
  \textstyle
  2\int_{\RT} \sigma_{\nu}^{2} |P_{v_{\xi}}| |v_{\xi}|
  \leq 
  2\int_{\RT} \sigma_{\nu}^{2} 
  |v_{\xi}| |R \otimes R \ (v_{\xi} \otimes v_{\xi})| 
  +
  4 \int_{\RT} \sigma_{\nu}^{2} 
  |v_{\xi}| |R \otimes R  \ (v_{\xi} \otimes w)| 
   =:  I + II.
\end{equation*}
With computations similar to those of the first step, using
the boundedness of the Riesz transform and the CKN
inequality, we obtain
\begin{equation}\label{FirstInBis}
  I \leq 
  \frac{\dot{B}_{\bar{s},\nu}}{8} 
  + Z \dot{B}_{\bar{s},\nu} a_{\nu},
\end{equation}
and, by possibly increasing the value of $Z$ at each step,
\begin{equation}\label{SecInBis}
  II\le
  Z\|   w  \|_{L^{4}}
  \| \sigma_{\nu}^{1/2}  v_{\xi}  \|^{1/4}_{L^{2}} 
  \| \sigma_{\nu}^{1/2}  \nabla  v_{\xi}  \|^{7/4}_{L^{2}}
  =
  Z
  \|   w  \|_{L^{4}}
  a_{\nu}^{1/8}\dot{B}_{\bar{s},\nu}^{7/8}
  \le
  \frac{\dot{B}_{\bar{s},\nu}}{8} 
  +
  \|  w  \|^{8}_{L^{4}}
  + Za_{\nu} \dot{B}_{\bar{s},\nu}.
\end{equation}
Next we have
\begin{equation}\label{ThirdInBis}
  |\xi|\int_{\RT} \sigma_{\nu}^{2} |v_{\xi}|^{2} 
  = 
  |\xi| \| \sigma_{\nu}  v_{\xi}   \|_{L^{2}}^{2}
  \le
  Z
  |\xi|  \| \sigma_{\nu}^{1/2}   \nabla v_{\xi}   \|_{L^{2}}
  =
  Z
  |\xi|
  (\dot{B}_{\bar{s},\nu} a_{\nu})^{1/2} 
  \le |\xi|^{2} + 
  Z
  \dot{B}_{\bar{s},\nu} a_{\nu};
\end{equation}
and
\begin{equation}\label{FourthInBis}
  \int_{\RT} \sigma_{\nu}^{2} |v_{\xi}|^{3} 
  = 
  \| \sigma_{\nu}^{2/3} v_{\xi} \|_{L^{3}}^{3}
  \le
  Z
  \| \sigma_{\nu}^{1/2}  \nabla v_{\xi}  \|^{2}_{L^{2}}
  \| \sigma_{\nu}^{1/2}  v_{\xi}  \|_{L^{2}}
  =
  Z
  \dot{B}_{\bar{s},\nu} a_{\nu}^{1/2} 
  \le
  \frac{\dot{B}_{\bar{s},\nu}}{8}  
  +Z
  \dot{B}_{\bar{s},\nu}a_{\nu}.
\end{equation}
Finally, for the perturbative terms we have
\begin{equation}\label{SixthInBis}
\begin{split}
  \textstyle
  3 \int_{\RT}\sigma_{\nu}^{2} |v_{\xi}|^{2} |w_{\xi}| 
  \le&
  3 \|   w_{\xi}  \|_{L^{4}}
  \| \sigma_{\nu}  v_{\xi}  \|_{L^{8/3}}^{2}
  \le
  Z
  \|  w_{\xi}  \|_{L^{4}}
  \| \sigma_{\nu}^{1/2}  v_{\xi}  \|^{1/4}_{L^{2}} 
  \| \sigma_{\nu}^{1/2}  \nabla v_{\xi}  \|^{7/4}_{L^{2}}
  \\
  =&
  Z
  \|  w_{\xi}  \|_{L^{4}}
  a_{\nu}^{1/8}\dot{B}_{\bar{s}, \nu} ^{7/8}
  \le
  \frac{\dot{B}_{\bar{s},\nu}}{8}
  +
  \|  w_{\xi} \|^{8}_{L^{4}}  
  + 
  Z\dot{B}_{\bar{s}, \nu} a_{\nu},
\end{split}
\end{equation}
and
\begin{equation}\label{SeventhInBis}
\begin{split}
  \textstyle
  18 \int_{\RT} \sigma_{\nu} |v_{\xi}| |\nabla v_{\xi}| |w_{\xi}| 
  \le &
  18  \| \sigma_{\nu}^{1/2}  \nabla v_{\xi}  \|_{L^{2}}
  \|  w_{\xi}  \|_{L^{4}} 
  \| \sigma_{\nu}^{1/2}   v_{\xi}   \|_{L^{4}}
  \\
  \le&
  Z
  \| \sigma_{\nu}^{1/2}  \nabla v_{\xi}  \|_{L^{2}}
  \| w_{\xi}\|_{L^{4}}
  \| \sigma_{\nu}^{1/2}   \nabla v_{\xi}  \|^{3/4}_{L^{2}}
  \| \sigma_{\nu}^{1/2}  v_{\xi}  \|^{1/4}_{L^{2}}
  \\
  =&
  Z
  \|  w_{\xi}  \|_{L^{4}}
  a_{\nu}^{1/8}\dot{B}_{\bar{s}, \nu} ^{7/8}
  \le
  \frac{\dot{B}_{\bar{s},\nu}}{8}
  +
  \|  w_{\xi} \|^{8}_{L^{4}}  
  + 
  Z\dot{B}_{\bar{s}, \nu} a_{\nu},
\end{split}
\end{equation}
We now plug the previous inequalities in (\ref{PutIn2})
and we obtain
\begin{equation*}
\begin{split}
  a_{\nu}(t)
  \psi_{\nu}(t)
  &
  \textstyle
  -a_{\nu}(\bar{s})
  +2\int_{\bar{s}}^{t}
  \dot{B}_{\bar{s},\nu}(s)\psi_{\nu}(s)ds
  \le
  \\
  \le&
  \textstyle
  \int_{\bar{s}}^{t}
  \psi_{\nu}(s)
  [\frac58\dot{B}_{\bar{s},\nu}(s)
  +6Z \dot{B}_{\bar{s}, \nu} a_{\nu}
  +|\xi|^{2}
  + 3\|  w_{\xi} \|^{8}_{L^{4}} 
  -k\dot{B}_{\bar{s}, \nu} a_{\nu}
  ]ds.
\end{split}
\end{equation*}
We subtract the first term at the right hand side
from the left hand side;
then we choose $k=6Z$ and note that
\begin{equation*}
  \textstyle
  \int_{\bar{s}}^{t}
  \dot{B}_{\bar{s},\nu}\psi_{\nu}
  \equiv
  -\frac{1}{6Z}
  \int_{\bar{s}}^{t}
  \dot{\psi}_{\nu}=
  \frac{\psi_{\nu}(\bar{s})-\psi_{\nu}(t)}{6Z}
  =
  \frac{1-\psi_{\nu}(t)}{6Z}
\end{equation*}
so that, for $\bar{s}\le t\le \bar{s}+T$, we obtain
\begin{equation}\label{LastIn}
  \textstyle
  a_{\nu}(t)
  \psi_{\nu}(t)
  -a_{\nu}(\bar{s})
  +
  \frac{1-\psi_{\nu}(t)}{6Z}
  \le
  |\xi|^{2} \int_{\bar{s}}^{t}
  \psi_{\nu}(s)ds+
  3 \int_{\bar{s}}^{t}
  \|  w_{\xi}(s,\cdot) \|^{8}_{L^{4}} ds.
\end{equation}
Consider now the increasing function, for $t\ge \bar{s}$,
\begin{equation}\label{eq:defBs}
  B_{\bar{s}}(t) := 
  \int_{\bar{s}}^{ t}
  \int_{\mathbb{R}^{3}} 
  |y|^{-1}|\nabla v_{\xi}(s,y)|^{2}dy ds
\end{equation}
which may become infinite at some point $t=t_{0}>\bar{s}$.
By the definition of $\bar{s}$, we know that 
$B_{\bar{s}}(t)\ge M$ for $t\ge \bar{s}+T/M$; 
since $B_{\bar{s},\nu}\to B_{\bar{s}}$ pointwise as
$\nu\to0$, we have also
\begin{equation*}
  \textstyle
  B_{\bar{s},\nu}(s)\ge \frac M2
  \quad\text{for}\quad 
  s\ge \bar{s}+\frac TM
  \quad\text{and}\quad 
  \nu
  \quad\text{small enough}.
\end{equation*}
Using this estimate for $s\ge \bar{s}+T/M$ and the obvious one
$B_{\bar{s},\nu}\ge0$ for $s\le \bar{s}+T/M$, we have easily
\begin{equation}\label{eq:estintpsi}
  \textstyle
  \int_{\bar{s}}^{\bar{s} + T} \psi_{\nu} (s) \ ds =
  \int_{\bar{s}}^{\bar{s} + T} e^{- 3Z B_{\bar{s},\nu}(s)} \ ds
  \le
  \frac{T}{M} + e^{-3ZM} \left( T - \frac{T}{M} \right)
  \le
  \frac{2T}{M}
\end{equation}
(here we assumed $Z\ge1$). We now use
the estimate
$a(\bar{s}) \le Ze^{4M^{2}}(\theta_{2}\epsilon)^{2}$
(proved in \eqref{BoundBySmallnessBis}) and note that
we can assume
\begin{equation}\label{eq:cond1}
  \theta_{2}\epsilon\le1
  \quad\implies\quad
  a(\bar{s}) 
  \le
  Ze^{4M^{2}}\theta_{2}\epsilon.
\end{equation}
Moreover by \eqref{BoundOnW} we have also
\begin{equation*}
  \|w_{\xi}\|_{L^{8}_{t}L^{4}_{x}}^{8} 
  =
  \|w\|_{L^{8}_{t}L^{4}_{x}}^{8} 
  \le Z \theta_{1} \epsilon
\end{equation*}
so that inequality \eqref{LastIn} implies
\begin{equation*}
  \textstyle
  (a_{\nu}(t)-\frac{1}{6Z})\psi_{\nu}(t)
  +\frac{1}{6Z}
  -3Z \theta_{1}\epsilon
  -Ze^{4M^{2}}\theta_{2}\epsilon
  -2|\xi|^{2}\frac{T}{M}
  \le
  0
\end{equation*}
or equivalently
\begin{equation}\label{eq:Int2}
  \textstyle
  a_{\nu}(t)
  +(\frac{1}{6Z}
  -3Z \theta_{1}\epsilon
  -Ze^{4M^{2}}\theta_{2}\epsilon
  -2|\xi|^{2}\frac{T}{M})
  e^{6Z B_{\bar{s},\nu}(t)}
  \le
  \frac{1}{6Z}.
\end{equation}
We now assume $\epsilon$ is so small that
\begin{equation}\label{eq:smallep}
  \textstyle
  3Z \theta_{1}\epsilon\le \frac{1}{30Z},
  \qquad
  Z e^{4M^{2}} \theta_{2}\epsilon\le \frac{1}{30Z},
\end{equation}
(this implies also \eqref{eq:cond1} and \eqref{eq:secondZ}),
so that \eqref{eq:Int2} implies
\begin{equation}\label{eq:Int3}
  \textstyle
  a_{\nu}(t)
  +(\frac{1}{10Z}
  -2|\xi|^{2}\frac{T}{M})
  e^{6Z B_{\bar{s},\nu}(t)}
  \le
  \frac{1}{6Z}.
\end{equation}
Assume in addition that $\xi$ satisfies
\begin{equation}\label{eq:condxi2}
  \textstyle
  (\frac{1}{10Z}
    -2|\xi|^{2}\frac{T}{M})>0
  \quad {\it i.e.} \quad 
  |\xi|^{2}T<\frac{M}{20Z}.
\end{equation}
Note that this condition is stronger than the first condition
\eqref{ResrtCondFirst} on $\xi$, {\it i.e.}
$|\xi|^{2}\bar{s}\le M^{2}$, since $M,Z\ge1$ and $\bar{s}\le T$.
Then, if we let $\nu\to0$, we have
\begin{equation*}
  \textstyle
  a_{\nu}(t)\to a(t):=
  \int_{\mathbb{R}^{3}} |y|^{-1}|v_{\xi}(t,y)|^{2}dy,
  \quad
  B_{\bar{s}, \nu}(t) \to
  B_{\bar{s}}(t):=
  \int_{\bar{s}}^{t} \int_{\mathbb{R}^{3}} 
  |y|^{-1} |\nabla v_{\xi}(s,y)|^{2}dyds
\end{equation*}
and \eqref{eq:Int3} implies,
for all $\bar{s}\le t\le \bar{s}+T$
\begin{equation}\label{eq:Int4}
  \textstyle
  a(t)
  +(\frac{1}{10Z}
  -2|\xi|^{2}\frac{T}{M})
  e^{6Z B_{\bar{s}}(t)}
  \le
  \frac{1}{6Z}.
\end{equation}
In particular we see that $a(t)$ and $B_{\bar{s}}(t)$ are
finite for $\bar{s}\le t\le \bar{s}+T$. Since by the definition of
$\bar{s}$ we already know that $B(\bar{s})\le 2M^{2} < + \infty$,
we conclude that 
\begin{equation*}
  B(s)< + \infty
  \quad\text{for all}\quad 
  0\le s\le \bar{s}+T.
\end{equation*}
In particular we have
\begin{equation}
  \textstyle
  B(T)=
  \int_{0}^{T}\int|y|^{-1}|\nabla v_{\xi}(s,y)|^{2}dyds
  \equiv
  \int_{0}^{T}\int|x-s \xi|^{-1}|\nabla v(s,x)|^{2}dyds
  < + \infty
\end{equation}
and then the same argument used to conclude the proof in the
first case ($\bar{s}=T$) gives also in the
second case ($\bar{s}<T$) that $L(T,\xi)$
is a regular set, provided
\eqref{eq:smallep}, \eqref{eq:condxi2} are satisfied.

\subsection{Conclusion of the proof}

Summing up, we have proved that there exists a universal
constant $Z$ such that for any $\widetilde{p}\in[2,4)$,
$M\ge1$, $T>0$ and 
$\xi\in \mathbb{R}^{3}\setminus0$ the following holds:
if $\epsilon=[u_{0}]_{\widetilde{p}}$ 
is small enough to satisfy \eqref{eq:smallep},
and $T,\xi$ are such that \eqref{eq:condxi2} holds,
then the segment $L(T,\xi)$ is a regular set
for the weak solution $u$.

Now define
\begin{equation*}
  \delta=\frac{1}{90Z^{2}}.
\end{equation*}
Then \eqref{eq:smallep} is implied by
\begin{equation}\label{eq:smallepdel}
  \theta_{1}\epsilon\le \delta,
  \qquad
  \theta_{2}\epsilon\le \delta e^{-4M^{2}}
\end{equation}
while \eqref{eq:condxi2} is implied by
\begin{equation*}
  \textstyle
  |\xi|^{2}T<M \delta
  \quad\iff\quad
  T>\frac{|T \xi|^{2}}{M \delta}
\end{equation*}
or equivalently
\begin{equation}\label{eq:}
  \textstyle
  (T,T \xi)\in\Pi_{M \delta},
  \qquad
  \Pi_{M \delta}:=
  \{(t,x)\in \mathbb{R}^{+}\times \mathbb{R}^{3}
  \colon t>\frac{|x|^{2}}{M \delta}\}.
\end{equation}
In other words, if $\epsilon$ satisfies 
\eqref{eq:smallepdel} and $(T,T \xi)$ belongs
to the paraboloid $\Pi_{M \delta}$, then
$L(T,\xi)$ is a regular set. Since $\Pi_{M \delta}$
is the union of such segments for arbitrary $T>0$,
we conclude that
$\Pi_{M \delta}$ is a regular set for the solution $u$,
provided \eqref{eq:smallepdel} holds.


\begin{thebibliography}{10}


\bibitem{CaffarelliKohnNirenberg84-a}
L. A. Caffarelli, R. Kohn and L. Nirenberg.
\newblock First order interpolation inequalities with weights.
\newblock {\em Compositio Math.}, 53(3):259--275, 1984.





\bibitem{CKN} 
L. A. Caffarelli, R. Kohn and L. Nirenberg. 
\newblock Partial regularity of suitable weak solutions of the Navier--Stokes equations. 
\newblock {\em Comm. Pure Appl. Math.}, 35:771--831, 1982.


\bibitem{Calderon} C. P. Calderon.
\newblock Existence of weak solutions for the Navier--Stokes equations
with initial data in $L^{p}$. 
\newblock {\em Trans. Amer. Math. Soc.}, 318(1):179--200, 1990.



\bibitem{Cannon} M. Cannone.
\newblock A generalization of a theorem by Kato on Navier--Stokes equations. 
\newblock {\em Rev. Mat. Iberoam.}, 13:515-541, 1997.




\bibitem{ChoOzawa09-a}
Y. Cho and T. Ozawa.
\newblock Sobolev inequalities with symmetry.
\newblock {\em Comm. Contemp. Math.}, 11(3):355--365, 2009.



\bibitem{Cordoba}
A. C\'{o}rdoba.
\newblock Singular integrals and maximal functions: the disk multiplier revisited.
\newblock arXiv:1310.6276.



\bibitem{CordobaFefferman}
A. C\'{o}rdoba and C. Fefferman.
\newblock A weighted norm inequality for singular integrals.
\newblock {\em Studia Math.}, 57(1):97--101, 1976.





\bibitem{DanconaCacciafesta11-a}
P. D'Ancona and F. Cacciafesta.
\newblock Endpoint estimates and global existence for the nonlinear Dirac equation with potential.
\newblock  {\em J. Diff. Eq.}, 254(5):2233--2260, 2013. 



\bibitem{DL} 
P. D'Ancona and R. Luc\`a.
\newblock Stein--Weiss and Caffarelli--Kohn--Nirenberg inequalities with higher
angular integrability.
\newblock {\em J. Math. Anal. App.}, 388(2):1061--1079, 2012.









\bibitem{DenapoliDrelichmanDuran10-a}
 P.~L. De~N{\'a}poli, I. Drelichman and R.~G. Dur{\'a}n.
\newblock Improved Caffarelli--Kohn--Nirenberg and trace inequalities for radial functions. 
\newblock {\em  Comm. Pure Appl. Anal.}, 11(5):1629--1642, 2012.  





\bibitem{Esc}
L. Escauriaza, G. Seregin and V. Sverak
\newblock Backward uniqueness for parabolic equations.
\newblock {\em Arch. Ration. Mech. Anal.}, 169:147--157, 2003.



\bibitem{Fabes}
E. Fabes, B. Jones and N. Riviere
\newblock The initial value problem for the Navier--Stokes equation with data in $L^{p}$.
\newblock {\em Arch. Ration. Mech. Anal.}, 45:222--240, 1972.



\bibitem{FangWang08-a}
D. Fang and C. Wang.
\newblock Weighted {S}trichartz estimates with angular regularity and their applications.
\newblock {\em Forum Math.}, 23:181--205, 2011





\bibitem{Gall} I. Gallagher, D. Iftimie and F. Planchon.
\newblock Asymptotics and stability for global solutions to the Navier--Stokes
equations.
\newblock {\em Ann. Inst. Four.}, 53, 5:1387--1424, 2003.





\bibitem{Giga} 
Y. Giga. 
\newblock Solutions for semilinear parabolic equations in $L^{p}$ and regularity of weak solutions of the Navier–Stokes system. 
\newblock {\em J. Diff. Eq.}, 62:186--212, 1986.

\bibitem{Gig3} 
Y. Giga and T. Miyakawa.  
\newblock Navier--Stokes flow in $\mathbb{R}^{3}$ with mesures
as initial vorticity and Morrey Spaces. 
\newblock {\em Comm. Part. Diff. Eq.}, 14:577--618, 1989.








\bibitem{Hopf} 
E. Hopf. 
\newblock Uber die Anfanqswertaufgabe f\"{u}r die hydrodynamischen Grundgleichungen. 
\newblock {\em Math. Nachr.}, 4:213--231, 1951.






\bibitem{Kato}
T. Kato. 
\newblock Strong $L^{p}$-solutions of the Navier--Stokes equation in $\mathbb{R}^{n}$, with applications to weak solutions. 
\newblock {\em Math. Z.}, 187: 471--480, 1984.






\bibitem{Tat} H. Koch and D. Tataru.
\newblock Well-posedness for the Navier--Stokes equations. 
\newblock {\em Adv. Math.}, 157(1): 22--35, 2001.









\bibitem{Lem}
P. G. Lemari{\'e}-Rieusset.
\newblock Recent developments in the Navier--Stokes problem.
\newblock CHAPMAN AND HALL/CRC. Research Notes in Mathematics Series 431, 2002.




\bibitem{Ler} 
J. Leray. 
\newblock Sur le mouvement d'un liquide visqueux emplissant l'espace. 
\newblock {\em Acta Math.}, 39:193--248, 1934.



\bibitem{Ren}
R. Luc{\`a}.
\newblock Regularity criteria with angular integrability for the Navier--Stokes equation.
\newblock {\em Nonlinear Anal.}, 105:24--40, 2014. 



\bibitem{MachiharaNakamuraNakanishi05-a}
S. Machihara, M. Nakamura, K. Nakanishi and T. Ozawa.
\newblock Endpoint {S}trichartz estimates and global 
solutions for the nonlinear {D}irac equation.
\newblock {\em J. Funct. Anal.}, 219(1):1--20, 2005.




\bibitem{Planc}
F. Planchon.
\newblock Global strong solutions in Sobolev or Lebesgue spces to
the incompressible Navier--Stokes equations in $\mathbb{R}^{3}$.
\newblock {\em Ann. Inst. Henry Poincare, Anal. Non Lineaire}, 13:319--336, 1996.




\bibitem{Prodi}
G. Prodi. 
\newblock Un teorema di unicit\`a per le equazioni
di Navier--Stokes.
\newblock{\em Ann. Mat. Pura Appl.}, 48(4):173--182, 1959.





\bibitem{Rogers}
T. Ozawa and K. M. Rogers
\newblock Sharp Morawetz estimates.
\newblock {\em J. Anal. Math.}, 121:163--175, 2013.






 



\bibitem{S2}
V. Scheffer.
\newblock Hausdroff measure and the Navier--Stokes equations.
\newblock {\em Comm. Math. Phys.}, 55(2):97--112, 1977.




\bibitem{Ser} 
J. Serrin. 
\newblock On the interior regularity of weak solutions of the Navier--Stokes equations. 
\newblock {\em Arch. Ration. Mech. Anal.}, 9:187--195, 1962.  



\bibitem{Ser2} 
J. Serrin. 
\newblock The initial value problem for the Navier-
Stokes equations.
\newblock {\em Nonlinear Problems (Proc. Sympos., Madison, Wis.)}, 
68--69, (Univ. of Wisconsin Press, Madison, Wis., 1963).  




\bibitem{Sohr} 
H. Sohr. 
\newblock Zur Regularit\"{a}tstheorie der instationaren Gleichungen von Navier--Stokes.
\newblock {\em Math. Z.}, 184:339--375, 1983.  


\bibitem{Stein}
E. M. Stein.
\newblock Note on singular integrals.
\newblock {\em Proc. Am. Math. Soc.}, 8:250--254, 1957.


\bibitem{Struwe} 
M. Struwe.
\newblock On partial regularity results for the Navier–Stokes equations.
\newblock {\em Comm. Pure Appl. Math.}, 41:437--458, 1988.



\bibitem{Tay}
M. E. Taylor 
\newblock Analysis on Morrey spaces and applications to Navier--Stokes and other evolution equations. 
\newblock {\em Comm. Part. Diff. Eq.}, 17(9-10):1407--1456, 1992. 
















\bibitem{Sterbenz05-a}
J. Sterbenz.
\newblock Angular regularity and {S}trichartz estimates for the wave equation.
\newblock {\em Int. Math. Res. Not.}, (4):187--231, 2005.
\newblock With an appendix by Igor Rodnianski.


\bibitem{Temam} 
\newblock R. T\'{e}mam. 
\newblock Navier--Stokes equations, Theory and Numerical Analysis. 
\newblock {\em North-Holland.} Amsterdam and New York, 1977.



\bibitem{WW}
W. von Wahl. 
\newblock Regularity of weak solutions of the Navier–Stokes equations. In Proc. 1983
Summer Inst. on Nonlinear Functional Analysis and Applications, Proceedings of Symposia
in Pure Mathematics, vol. 45, pp. 497--503 (Providence, RI: American Mathematical Society,
1989).










































   
























\end{thebibliography}
\end{document}